\documentclass[review]{elsarticle}

\usepackage{lineno,hyperref}
\usepackage{amsfonts}
\usepackage{amsmath}
\modulolinenumbers[5]

\usepackage{xcolor}         
\usepackage{algorithm}
\usepackage{algorithmic}

\journal{TBD}

\newcommand{\calA}{{\mathcal{A}}}
\newcommand{\calB}{{\mathcal{B}}}
\newcommand{\calC}{{\mathcal{C}}}

\newcommand{\calF}{{\mathcal{F}}}

\newcommand{\calM}{{\mathcal{M}}}
\newcommand{\calN}{{\mathcal{N}}}
\newcommand{\calO}{{\mathcal{O}}}
\newcommand{\calP}{{\mathcal{P}}}

\newcommand{\calS}{{\mathcal{S}}}
\newcommand{\calT}{{\mathcal{T}}}
\newcommand{\calU}{{\mathcal{U}}}
\newcommand{\calV}{{\mathcal{V}}}

\def\tcr{\color{red}}

\newcommand{\N}{{\mathbb{N}}} 
\newcommand{\R}{{\mathbb{R}}} 

\newtheorem{theorem}{Theorem} 
\newtheorem{lemma}[theorem]{Lemma}
\newtheorem{corollary}[theorem]{Corollary}
\newdefinition{definition}{Definition}
\newproof{proof}{Proof}

\begin{document}

\begin{frontmatter}

\title{Binary sequences with a Ces\`aro limit}

\author{Jonathan M. Keith}
\author{Greg Markowsky}
\address{School of Mathematics, Monash University, Wellington Rd, Clayton VIC 3800, Australia}

\begin{abstract}
The Ces\`aro limit - the asymptotic average of a sequence of real numbers - is an operator of fundamental importance in probability, statistics and mathematical analysis. To better understand sequences with Ces\`aro limits, this paper considers the space $\calF$ comprised of all binary sequences with a Ces\`aro limit, and the associated functional $\nu: \calF \rightarrow [0,1]$ mapping each such sequence to its Ces\`aro limit. The basic properties of $\calF$ and $\nu$ are enumerated, and chains (totally ordered sets) in $\calF$ on which $\nu$ is countably additive are studied in detail. The main result of the paper concerns a structural property of the pair $(\calF,\nu)$, specifically that $\calF$ can be factored (in a certain sense) to produce a monotone class on which $\nu$ is countably additive. In the process, a slight generalisation and clarification of the monotone class theorem for Boolean algebras is proved. 
\end{abstract}


\end{frontmatter}


\section{Introduction} \label{intro}

Consider limits of the form
\[
\lim_{N \rightarrow \infty} \frac{1}{N} \sum_{n=1}^N x_n
\]
where $x_n \in \mathbb{R}$ for $n \in \N$ (here and throughout this paper, $\N$ denotes the natural numbers excluding 0). These are known as Ces\`aro limits (see \cite{bishop2014} for example) or sometimes as Ces\`aro means or Ces\`aro averages (as in \cite{crismale2017}), and arise naturally in multiple mathematical fields, including statistics, probability, functional and general analysis, and in the study of stochastic processes, particularly ergodic processes. They are important in many applications (see list of references in~\cite{bishop2014}). They are named for mathematician Ernesto Ces\`aro (1859-1906), who was certainly not the first to consider the asymptotic properties of a sequence of averages, but used them to define a generalised limit for divergent series (see \cite{ferraro1999first}, and the references therein, for an interesting historical account).


As a prelude to a more detailed study of sequences with Ces\`aro limits, this paper focuses on binary sequences. The collection $\calF$ of subsets of $\N$ that induce binary sequences with a Ces\`aro limit is defined in Section~\ref{F_and_nu}, in tandem with a set function $\nu$ that maps such subsets to the corresponding Ces\`aro limit. The basic properties of $\calF$ and $\nu$ are enumerated in that section. It turns out that $\nu$ has many of the properties of a finitely additive measure, also known as a {\em charge}. However, $\calF$ is not a field, and thus $\nu$ is not a charge unless restricted to a field of sets contained in $\calF$. Section~\ref{F_and_nu} also introduces the collection of {\em null sets} $\calN$, comprised of subsets of $\N$ that induce binary sequences with zero Ces\`aro limits.


Although $\nu$ is not countably additive on $\calF$, it turns out that chains (totally ordered sets) in $\calF$ on which $\nu$ is countably additive have a number of useful properties. Two sections of the paper are devoted to exploring the properties of such chains. Section~\ref{uniform_convergence_section} characterises such chains in terms of uniform convergence to Ces\`aro limits. Section~\ref{null_modification_section} develops a construction that is here called a {\em null modification}. This construction modifies the elements of a chain of sets in $\calF$ by adding and removing null sets to produce a new chain on which $\nu$ is countably additive. This section makes frequent reference to Boolean algebras and their quotients: a brief review of this topic is therefore included in Section~\ref{Boolean_review}, with special attention to the Boolean quotient $\calP(\N) / \calN$. 

Section~\ref{quotient_space} considers the space $[\calF]$ - the image of $\calF$ under the quotient map $\xi : \calP(\N) \rightarrow \calP(\N) / \calN$. The collection of equivalence classes $[\calF]$ is shown to be a monotone class in the Boolean quotient $\calP(\N) / \calN$. This is a useful insight into the structure of $\calF$ because it implies that every field of sets in $\calF$ can be extended in such a way that the extension maps to a countably complete subalgebra of $[\calF]$, as a consequence of the monotone class theorem for Boolean algebras, which is reviewed in Section~\ref{Boolean_review}. This version of the monotone class theorem is an abstraction of the well known version for fields of sets, and is a slight generalisation of similar results in the literature. Consequently, a full proof is presented, highlighting and clarifying several subtleties that are important for a proper understanding of the abstraction. 

\section{Ces\`aro limits of binary sequences} \label{F_and_nu}

Let $\calP(X)$ denote the power set of an arbitrary set $X$ and consider the following definitions.

\begin{definition}
\label{calFdefn}
For any $A \in \calP(\N)$, define a {\em partial average}
\[
\nu_N(A) := \frac{1}{N} \sum_{n=1}^N I_A(n),
\] 
for each $N \in \N$, where $I_A$ is the indicator function for the set $A$. 
\end{definition}

\begin{definition} \label{upperandlower}
For any $A \in \calP(\N)$, define the {\em upper and lower Ces\`aro limits} respectively as
\[
\nu^+(A) := \limsup_{N \rightarrow \infty} \nu_N(A) \mbox{ and } \nu^-(A) := \liminf_{N \rightarrow \infty} \nu_N(A).
\]
\end{definition}

Naturally, the sets upon which $\nu^+$ and $\nu^-$ coincide are of particular interest, motivating the following definition.

\begin{definition}
Let $\calF$ be the collection of subsets $A \subseteq \N$ such that
\[
\nu(A) := \lim_{N \rightarrow \infty} \nu_N(A)
\] 
exists in the interval $[0,1]$. That is, $\calF$ is the collection of subsets $A$ of $\N$ for which the {\em Ces\`aro limit} of the binary sequence $x_n := I_A(n)$ (for each $n \in \N$) exists, and $\nu(A)$ is that limit for any $A \in \calF$.
\end{definition}

Note that 
\begin{enumerate}
\item $0 \leq \nu^-(A) \leq \nu^+(A) \leq 1$,  
\item $\nu^-(A) = \nu^+(A) \iff A \in \calF$ and 
\item $A \in \calF \implies \nu(A) = \nu^+(A)$.
\end{enumerate}

While $\nu_N$ is a measure on the power set of $\N$, being simply a scaling of counting measure on a finite set, it is evident that $\nu$ is not; to see this, note that any singleton set $\{k\}$ will have $\nu(\{k\})=0$, but $\cup_{k \in \N} \{k\} = \N$, so that $\nu$ is not countably additive. It is, however, finitely additive, and is thus a charge when restricted to fields of sets contained in $\calF$. It may be helpful to consider specific examples of sets for which $\nu$ exists. If $D_m$ is the set of multiples of an integer $m$, then $\nu(D_m) = \frac{1}{m}$. The same holds if $D^r_m$ is the set of all numbers equal to $r$ modulo $m$, and by taking unions of such sets one can obtain a set with any rational number as its charge. To obtain an irrational number $s$ as a charge is only slightly harder, and it can be achieved by the following algorithm. Start with $A_1 := \{1\}$, then for $N \geq 2$, if $\frac{1}{N} \sum_{n=1}^N I_{A_N}(n)<s$ let $A_{N+1} := A_N \cup\{N+1\}$, else let $A_{N+1} := A_N$. It is straightforward to show that if $A := \cup_{n \in \N} A_n$, then $\nu(A) = s$. 

The following method of describing sets is also useful for constructing specific examples. Let $z_n$ be a sequence of positive integers for $n \geq 2$, and let $z_1$ be a non-negative integer. Let $Z_n = \sum_{j=1}^n z_n$. Then let $A$ be defined by
\[
I_A(n):=\begin{cases}
1 & \mbox{ if $Z_{2k-1}+1 \leq n \leq Z_{2k}$ for some $k\geq 1$}  \\
0 & \mbox{if $n \leq z_1$ or $Z_{2k} + 1 \leq n \leq Z_{2k+1}$ for some $k \geq 1$}
\end{cases}.
\]
In words, $I_A(x_n)$ is $z_1$ zeroes, followed by $z_2$ ones, followed by $z_3$ zeroes, etc. It may be checked that $\nu^+(A) = \limsup_{N \to \infty} \nu_{Z_{2N}}(A) = \limsup_{N \to \infty} \frac{\sum_{j=1}^N z_{2j}}{Z_{2N}}$ and $\nu^-(A) = \liminf_{N \to \infty} \nu_{Z_{2N-1}}(A) = \liminf_{N \to \infty} \frac{\sum_{j=1}^{N-1} z_{2j}}{Z_{2N-1}}$. To form a simple example of a set which is not in $\calF$, let $z_n = 2^{n-1}$. Then $Z_n = 2^n-1$ and $\sum_{j=1}^N z_{2j} = \sum_{j=1}^N 2^{2j-1} = \frac{2}{3}(2^{2N} -1)$. Furthermore, 
\[
\nu^+(A) = \limsup_{N \to \infty} \frac{\frac{2}{3}(2^{2N} -1)}{2^{2N}-1} = \frac{2}{3}
\]
and \[
\nu^-(A) = \liminf_{N \to \infty} \frac{\frac{2}{3}(2^{2(N-1)} -1)}{2^{2N-1}-1} = \frac{1}{3}.
\]

This set can then be used to construct two sets $B, C \in \calF$, such that  $B \cap C \notin \calF$, thereby showing that $\calF$ is not a field. Let $B$ be the set of all even numbers, and let $C$ be defined by
\[
I_C(n):=\begin{cases}
1 & \mbox{ if $n$ is even and $\frac{n}{2} \in A$, or $n$ is odd and $\frac{n+1}{2} \notin A$ }  \\
0 & \mbox{otherwise}
\end{cases}.
\]
It is clear $B \in \calF$, with $\nu(B) = \frac{1}{2}$, and the same conclusion follows for $C$ upon noting that exactly one of $\{2k-1,2k\}$ lies in $C$ for every $k \in \N$. However, $B \cap C = 2A$, the set of the doubles of elements of $A$, and therefore $B \cap C \notin \calF$.

For the purpose of intuition, it is profitable to think of sets in $\calF$ in this manner, as defined by the concatenation of alternating strings of zeroes and ones of variable length. The following lemma describes the lengths of these strings allowable for a set to be in $\calF$. Consider any $A \subseteq \N$. For each $N \in \N$, define $P_A(N)$ to be the smallest integer $k > 0$ such that $I_A(N+k) = 1$ or define $P_A(N) = \infty$ if no such integer exists. Similarly, define $Q_A(N)$ to be the smallest integer $k > 0$ such that $I_A(N+k) = 0$ or define $Q_A(N) = \infty$ if no such integer exists.

\begin{lemma} \label{useful_lemma}
Consider $A \subseteq \N$.
\begin{enumerate}
\item If $A$ contains $F \in \calF$ such that $\nu(F) > 0$, then $\nu^-(A) > 0$ and $P_A$ is $o(N)$.
\item If $A$ is contained in $F \in \calF$ such that $\nu(F) < 1$, then $\nu^+(A) < 1$ and $Q_A$ is $o(N)$.
\end{enumerate}
\end{lemma}

\begin{proof}
Suppose $A$ contains $F \in \calF$ such that $\nu(F) > 0$. Then $\nu^-(A) \geq \nu^-(F) = \nu(F) > 0$. Moreover $F$ contains infinitely many integers, so $P_A(N) \leq P_F(N) < \infty$ for all $N \in \N$. If $F$ excludes only finitely many integers then $P_A(N) = P_F(N) = 1$ for large enough $N$, making $P_A(N)$ trivially $o(N)$, so assume $F$ excludes infinitely many integers. There are thus infinitely many positive integers $N_1 < N_2 < \ldots$ such that $I_F(N_i + 1) \neq I_F(N_i )$. Note
\[
\nu(F) = \nu_{N_1}(A) + \sum_{i = 1}^{\infty} (\nu_{N_{i+1}}(A) - \nu_{N_i}(A))
\]
is a series consisting of alternating positive and negative terms corresponding respectively to runs of ones and zeros in the sequence $( I_F(1), I_F(2), \ldots )$. Since the series converges, terms corresponding to runs of zeros must decrease in magnitude to 0.

Now consider any $N$ with $P_F(N) = k+1 > 0$.  Then 
\[
\nu_{N+k}(F) - \nu_N(F) = \nu_N(F)\frac{N}{N+k} - \nu_N(F) = - \nu_N(F) \left( \frac{k}{N+k} \right).
\]
In particular, if $N_i$ corresponds to the end of a run of ones, then the subsequent run of zeros contributes a term
\[
- \nu_{N_i}(F) \left( \frac{P_F(N_i) - 1}{N+P_F(N_i) - 1} \right)
\]
to the above series, and since $\nu_{N_i}(F) \rightarrow \nu(F) > 0$, these terms can only go to 0 if 
\[
\frac{P_F(N_i) - 1}{N+P_F(N_i) - 1} \rightarrow 0
\]
implying $P_A(N) \leq P_F(N)$ is $o(N)$. The second part of the lemma follows by applying the first part to $A^c$. 
\qed \end{proof}

It can be shown that the term $o(N)$ in Lemma \ref{useful_lemma} cannot be replaced by $o(N^{1-\epsilon})$ for any $\epsilon > 0$, as follows. For positive integer $q$ let $z_n = n^q$, and form a set $A$ by the method described earlier in this section. Then, by comparing the sum with an integral, the easy estimates $\frac{N^{q+1}}{q+1} \leq Z_{N} \leq \frac{(N+1)^{q+1}}{q+1}$ are obtained, and since $\sum_{j=1}^N z_{2j} = 2^q \sum_{j=1}^N z_{j}$ it follows also that $\frac{2^q N^{q+1}}{q+1} \leq \sum_{j=1}^N z_{2j} \leq \frac{2^q (N+1)^{q+1}}{q+1}$. Thus, $\nu^+(A) = \limsup_{N \to \infty} \frac{\sum_{j=1}^N z_{2j}}{Z_{2N}} \leq \limsup_{N \to \infty} \frac{2^q (N+1)^{q+1}}{(2N)^{q+1}} = \frac{1}{2}$, and $\nu^-(A) = \liminf_{N \to \infty} \frac{\sum_{j=1}^{N-1} z_{2j}}{Z_{2N-1}} \geq \liminf_{N \to \infty} \frac{2^q (N-1)^{q+1}}{(2N-1)^{q+1}} = \frac{1}{2}$, so $A \in \calF$. However, $P_A(Z_{2N}) = (2N+1)^q + 1$, and $Z_{2N} \leq \frac{(2N+1)^{q+1}}{q+1}$, hence $\frac{P_A(Z_{2N})^{(q+1)/q}}{Z_{2N}}$ is bounded below (and above) by a constant for any $q > 0$, giving the result.

\if2
{\tcr Jon, do you know whether, if we consider sets as binary strings, what is the Lebesque measure of sets for which $\nu$ exists?} {\color{blue} I don't follow. I can see you could map an element of $\calF$ to a real number, by treating it as a binary representation, but that's a single point. But maybe you are thinking about the Lebesgue measure of the image under this transformation of a field of sets contained in $\calF$? I think it's possible to construct a field of sets for which the Lebesgue measure of its image is any desired value in $[0,1]$.}

{\tcr Also, is it worth mentioning that it's a set of full measure with the measure on paths coming from Polya's urn?} {\color{blue} I think so, can you expand on this? Perhaps this will be a useful point for the next paper, if that's about ergodic processes.}
\fi

Sets that $\nu$ maps to zero play an important role in this paper.

\begin{definition}
The {\em null sets} in $\calF$ are the elements of $\calN := \{ A \in \calF : \nu(A) = 0 \}$. 
\end{definition}

Null sets are easy to find. The set of square numbers is a null set, as is the set of cubes, etc. The set of powers of 2, or of any other base, is a null set. The set of primes is shown to be a null set by the prime number theorem. 

The next lemma, which describes key properties of $\calF$ and the Ces\'aro limits of its elements, is of fundamental importance in subsequent sections. 

\begin{lemma}
\label{calFproperties}
The collection $\calF$ and set functions $\nu^+$, $\nu^-$ and $\nu$ have the following properties.
\begin{enumerate}
\item $\emptyset, \N \in \calF$, with $\nu(\emptyset) = 0$ and $\nu(\N) = 1$.
\item For all $A \in \calF$, $A^c \in \calF$ with $\nu(A^c) = 1 - \nu(A)$.
\item For all $A, B \in \calF$, $A \cup B \in \calF$ if and only if $A \cap B \in \calF$, and if either is true then $\nu(A \cup B) = \nu(A) + \nu(B) - \nu(A \cap B)$.
\item If $A, B \in \calP(\N)$, then $\nu^+(A \cup B) \leq \nu^+(A) + \nu^+(B)$. If $A$, $B$, and $A \cup B$ are all in $\calF$, then $\nu(A \cup B) \leq \nu(A) + \nu(B)$.
\item For $A, B \in \calP(\N)$ such that $A \subseteq B$, 
\begin{enumerate}
\item $\nu^+(A) \leq \nu^+(B)$,  
\item $\nu^+(B \setminus A) \geq \nu^+(B) - \nu^+(A)$, 
\item $\nu^-(A) \leq \nu^-(B)$, and 
\item $\nu^-(B \setminus A) \leq \nu^-(B) - \nu^-(A)$. 
\end{enumerate}
If in addition $A, B \in \calF$, then 
\begin{enumerate}
\item $\nu(A) \leq \nu(B)$, 
\item $B \setminus A \in \calF$,
\item $\nu(B \setminus A) = \nu(B) - \nu(A)$, and
\item $\nu(A) = \nu(B) \iff B \setminus A \in \calN$.
\end{enumerate}
\item If $A \in \calN$, then any $B \subseteq A$ satisfies $B \in \calN$. Consequently, for any $C \in \calP(\N)$, 
\begin{enumerate}
\item $A \cap C, A \setminus C \in \calN$, and
\item $\nu^+(A \cup C) = \nu^+(C \setminus A) = \nu^+(C)$.
\end{enumerate}
If $C \in \calF$, then $A \cup C, C \setminus A \in \calF$.
\item For pairwise disjoint sets $A_1, \ldots, A_K \in \calF$, $A_1 \cup \ldots \cup A_K \in \calF$ with 
\[
\nu(A_1 \cup \ldots \cup A_K) = \sum_{k=1}^K \nu(A_k).
\] 
\item Consider $\calC \subseteq \calP(\N)$. Then 
\[
\nu^-(\bigcup \calC) \geq \sup\{ \nu^-(A) : A \in \calC \}.
\]
In particular, if $\{ A_k \}_{k=1}^{\infty} \subset \calF$ are pairwise disjoint, then
\[
\nu^-(\cup_{k=1}^{\infty} A_k) \geq \sum_{k=1}^{\infty} \nu(A_k).
\]
\item Consider a chain $\calC \subset \calP(\N)$ such that $\nu_N(A) \leq \nu^+(A)$ for all $N \in \N$ and $A \in \calC$. Then 
\begin{enumerate}
\item $\nu^+(\bigcup \calC) = \sup_{A \in \calC} \nu^+(A)$, and
\item if $\calC \subset \calF$, then $\bigcup \calC \in \calF$ with $\nu(\bigcup \calC) = \sup_{A \in \calC} \nu(A)$.
\end{enumerate}
In particular, if $\{ A_k \}_{k=1}^{\infty} \subset \calF$ are pairwise disjoint with $\nu_N(A_k) \leq \nu(A_k)$ for all $k, N \in \N$, then $\cup_{k=1}^{\infty} A_k \in \calF$ and $\nu(\cup_{k=1}^{\infty} A_k) = \sum_{k=1}^{\infty} \nu(A_k)$.
\if2 {\tcr Very interesting fact. I think it can be generalized somewhat, for instance is all but finitely many satisfy $\nu_N(A_k) \leq \nu(A_k)$, or if $\nu_N(A_k) \leq \nu(A_k)+ \epsilon_k$ where $\sum_{k=1}^\infty \epsilon_k < \infty$. But I suspect that the statement given probably covers the important cases, and thus probably best not to complicate it.} {\color{blue} (HHH One option is to split it into two properties: this simpler version and a more general one.)} \fi 
\item Consider $\calC \subseteq \calF$. If $\sup \{ \nu(A) : A \in \calC \} = 1$, then $\bigcup \calC \in \calF$ and $\nu(\bigcup \calC) = 1$. In particular, if $\{ A_k \}_{k=1}^{\infty} \subset \calF$ are pairwise disjoint with $\sum_{k=1}^{\infty} \nu(A_k) = 1$, then $\cup_{k=1}^{\infty} A_k \in \calF$ and $\nu(\cup_{k=1}^{\infty} A_k) = 1$.
\end{enumerate}
\end{lemma}

\begin{proof}
Property 1 is trivial. For Property 2, consider any $A \in \calF$ and note 
\[
\lim_{N \rightarrow \infty} \nu_N(A^c) = \lim_{N \rightarrow \infty} \frac{1}{N} \sum_{n=1}^N (1 - I_{A}(n)) 
= 1 - \lim_{N \rightarrow \infty} \nu_N(A)
\]
hence $A^c \in \calF$ and $\nu(A^c) = 1 - \nu(A)$.

For 3, consider $A,B \in \calF$ and note $I_{A \cup B} = I_A + I_B - I_{A \cap B}$, hence
\[
\lim_{N \rightarrow \infty} \nu_N(A \cup B) = \nu(A) + \nu(B) - \lim_{N \rightarrow \infty} \nu_N(A \cap B)
\]
if either limit exists, and the statement follows immediately.

For 4, note $I_{A \cup B} \leq I_A + I_B$, hence $\nu_N(A \cup B) \leq \nu_N(A) + \nu_N(B)$ for all $N \in \N$. The first statement then follows by taking the $\limsup$ and the second by taking limits as $N \rightarrow \infty$. (The second statement alternatively follows from 3).

For 5, note $I_A \leq I_B$, hence $\nu^+(A) \leq \nu^+(B)$ and $\nu^-(A) \leq \nu^-(B)$. Also $I_B = I_A + I_{B \setminus A}$, hence 
\[
\limsup_{N \rightarrow \infty} \nu_N(B) \leq \limsup_{N \rightarrow \infty} \nu_N(A) + \limsup_{N \rightarrow \infty} \nu_N(B \setminus A).
\]
That is, $\nu^+(B \setminus A) \geq \nu^+(B) - \nu^+(A)$. That $\nu^-(B \setminus A) \leq \nu^-(B) - \nu^-(A)$ is shown similarly. If $A, B \in \calF$, then $\nu(A) = \nu^+(A) \leq \nu^+(B) = \nu(B)$ and
\[
\lim_{N \rightarrow \infty} \nu_N(B \setminus A) = \lim_{N \rightarrow \infty} \nu_N(B) - \lim_{N \rightarrow \infty} \nu_N(A) = \nu(B) - \nu(A) 
\]
thus $B \setminus A \in \calF$ and $\nu(B \setminus A) = \nu(B) - \nu(A)$. Hence $\nu(B) - \nu(A) = 0 \iff \nu(B \setminus A) = 0$. 

For 6, note for every $N \in \N$,
\[
0 \leq \nu_N(B) \leq \nu_N(A)
\]
and $B \in \calN$ follows by letting $N \rightarrow \infty$. The consequences 6(a) follow because $A \cap C \subseteq A$ and $A \setminus C \subseteq A$. For 6b, note
\[
\nu^+(C) = \nu^+(C) - \nu^+(A) \leq \nu^+(C \setminus A) \leq \nu^+(A \cup C) \leq \nu^+(A) + \nu^+(C) = \nu^+(C).
\]
If $C \in \calF$, $A \cup C \in \calF$ by 3, and $C \setminus A \in \calF$ by 5.

For 7, take Property 3 with $A \cap B = \emptyset$, to conclude that $A \cup B \in \calF$ and $\nu(A \cup B) = \nu(A) + \nu(B)$. The property then follows by induction.

For 8, by Property~5, $\nu^-(\bigcup \calC) \geq \nu^-(A)$ for all $A \in \calC$. Hence $\nu^-(\bigcup \calC) \geq \sup\{ \nu^-(A) : A \in \calC \}$. The second part of 8 follows by defining $B_j := \cup_{k=1}^j A_k$, so that by Property~7, $B_j \in \calF$ with $\nu(B_j) = \sum_{k=1}^j \nu(A_k)$ for each $j \in \N$. Then apply the first part of 8 to $\calC := \{ B_j \}_{j=1}^{\infty}$.

For 9, note for any $N \in \N$, $\nu_N(\bigcup \calC) = \sup_{A \in \calC} \nu_N(A) \leq \sup_{A \in \calC} \nu^+(A)$, hence $\nu^+(\bigcup \calC) \leq \sup_{A \in \calC} \nu^+(A)$. Moreover, for any $\epsilon > 0$, one can choose $C \in \calC$ such that $\nu^+(\bigcup \calC) \geq \nu^+(C) > \sup_{A \in \calC} \nu^+(A) - \epsilon$. Letting $\epsilon \rightarrow 0$ gives $\nu^+(\bigcup \calC) = \sup_{A \in \calC} \nu^+(A)$. If $\calC \subset \calF$, then by 8,
\[
\nu^-(\bigcup \calC) \geq \sup_{A \in \calC} \nu(A) = \nu^+(\bigcup \calC),
\]
hence $\bigcup \calC \in \calF$ with $\nu(\bigcup \calC) = \sup_{A \in \calC} \nu(A)$. The last part of 9 follows by setting $\calC := \{ B_j \}_{j=1}^{\infty}$ as defined in the proof of 8. Then 9b gives $\cup_{k=1}^{\infty} A_k = \cup_{k=1}^{\infty} B_k \in \calF$ with $\nu(\cup_{k=1}^{\infty} A_k) = \sup_k \nu(B_k) = \sum_{k=1}^{\infty} \nu(A_k)$.

For 10, note that $1 \geq \nu^+(\bigcup \calC) \geq \nu^-(\bigcup \calC) \geq \sup\{ \nu(A) : A \in \calC \} = 1$, using Property~8. Hence $\nu^+(\bigcup \calC) = \nu^-(\bigcup \calC)$ and the first part follows. The second part of 10 follows by applying the first part to $\calC = \{ B_j \}_{j=1}^{\infty}$ defined above in the proof of 8.
\qed \end{proof}

Property 9 is particularly important in what follows, so it may be helpful to discuss an example. Let $\calO$ denote the set of odd numbers which are at least 3, and for $k \in \N \cup \{0\}$ set $D_k = \{2^k m: m \in \calO\}$. Then it may be checked that $\nu_N(D_k) \leq \nu(D_k)$ for all $N$ and $\nu(D_k) = \frac{1}{2^{k+1}}$. The set $D := \cup_{k \in \N} D_k$ is all of $\N$ with a null set (the powers of $2$) removed, and thus $\nu(D) = 1 = \sum_{k = 0}^ \infty \nu(D_k)$, so that Property 9 holds. If $1$ were included in $\calO$ the same conclusion would hold, even though the sufficient condition $\nu_N(D_k) \leq \nu(D_k)$ would not (it is evident, however, that this or some other condition is needed to ensure countable additivity, since $\nu$ is not a measure). This example will appear briefly again in Section \ref{null_modification_section}, which contains a method for modifying sets by removing a null set so that Property 9 can be applied.

\section{Uniform convergence of chains in $\calF$}
\label{uniform_convergence_section}

A real-valued function on any set generates a chain of subsets consisting of inverse images of rays in $\R$. Thus real-valued functions can potentially be manipulated by modifying such chains. In this section and Section~\ref{null_modification_section}, two analysis tools for studying chains in $\calF$ are developed. The first of these is a characterisation of a certain class of chains in $\calF$ in terms of uniform convergence of partial averages over the sets in the chain. 

The following notation is helpful to describe this characterisation. Consider a chain of sets $\calT \subset \calP(\N)$ (that is, a collection of sets that is totally ordered by set inclusion). Let $\calT_{\cup}$ and $\calT_{\cap}$ denote the closure of $\calT$ under unions and intersections, respectively. That is, $\calT_{\cup} := \{ \bigcup \calC : \calC \subseteq \calT \}$ and $\calT_{\cap} := \{ \bigcap \calC : \calC \subseteq \calT \}$. Also set $\calT_* := \calT_{\cup} \cup \calT_{\cap}$. Some basic properties of $\calT_*, \calT_{\cup}$, and $\calT_{\cap}$ are the following.

\begin{lemma} \label{basicfacts}
Suppose $\calT \subset \calP(\N)$ is a chain. Then
\begin{enumerate}
    \item $\calT_{\cup}$, $\calT_{\cap}$, and $\calT_*$ are chains.
    
    \item $\calT_*$ is closed under unions and intersections.
\end{enumerate}
\end{lemma}

\begin{proof}
To prove 1, first consider $\calT_{\cup}$. If $A \in \calT$ and $B = \bigcup \calC$ with $\calC \subseteq \calT$, then either $(i)$ $A \subseteq C$ for some $C \in \calC$, in which case $A \subseteq B$, or $(ii)$ $A \supseteq C$ for all $C \in \calC$, in which case $A \supseteq B$. Alternatively, suppose $A = \bigcup \calC_1$ and $B = \bigcup \calC_2$ with $\calC_1, \calC_2 \subseteq \calT$. By the previous argument, for every $C \in \calC_1$ either $C \subseteq \bigcup \calC_2$ or $C \supseteq \bigcup \calC_2$; if $C \subseteq \bigcup \calC_2$ for every $C \in \calC_1$ then $\bigcup \calC_1 \subseteq \bigcup \calC_2$, otherwise $C \supseteq \bigcup \calC_2$ for some $C \in \calC_1$, in which case $\bigcup \calC_1 \supseteq \bigcup \calC_2$. Thus $\calT_{\cup}$ is a chain. A complementary argument shows $\calT_{\cap}$ is a chain. To show $\calT_*$ is a chain, one must identify an ordering between $A = \bigcap \calC_1 \in \calT_{\cap}$ and $B=\bigcup \calC_2\in \calT_{\cup}$ with $\calC_1, \calC_2 \subseteq \calT$. If $C \supseteq D$ for every $C \in \calC_1, D \in \calC_2$, then $B \subseteq A$, otherwise $C \subseteq D$ for some $C \in \calC_1, D \in \calC_2$, in which case $A \subseteq B$. 

As for 2, suppose first that $\calC \subseteq \calT_{\cap}$ and consider $B = \bigcup \calC$. It is possible $B \in \calT_\cap$, but in this case there is nothing to prove since $\calT_\cap \subseteq \calT_*$, so assume $B \notin \calT_\cap$. Any $A \in \calC$ can be expressed as an intersection of sets in $\calT$, and if each of these sets contained $B$ then one would have $B \subseteq A$ and hence $B = A$, contradicting $B \notin \calT_\cap$. Hence there exists $A' \in \calT$ that contains $A$ but not $B$. Then $A \subseteq A' \subseteq B$, since $(\calT_{\cap})_{\cup}$ is a chain by the first part of this lemma. It follows that $B = \bigcup \{ A' \in \calT : A' \subseteq B \}$, and hence $B \in \calT_\cup$. This shows $(\calT_{\cap})_{\cup} \subseteq \calT_{\cap} \cup \calT_{\cup} = \calT_*$. Now suppose $\calC \subseteq \calT_*$. Then $\bigcup \calC = B_1 \cup B_2$ where $B_1 \in (\calT_{\cap})_{\cup} \subseteq \calT_*$ and $B_2 \in (\calT_{\cup})_{\cup} = \calT_{\cup} \subseteq \calT_*$. But $\calT_*$ is a chain, hence $B_1 \cup B_2$ is either $B_1$ or $B_2$. Either way $\bigcup \calC \in \calT_*$, hence $\calT_*$ is closed under unions. A complementary argument shows $\calT_*$ is closed under intersections.
\qed \end{proof}

The following theorem identifies three alternative characterisations of a class of well behaved chains in $\calF$. The first characterisation implies countable additivity of the restriction of $\nu$ to the chain: it thus identifies chains in $\calF$ on which $\nu$ behaves like a measure. The other characterisations identify other useful properties of such chains, in particular, uniform convergence of partial averages of elements of the chain.

\begin{theorem} \label{uniformconvergence}
Let $\calT \subset \calF$ be a chain of sets. Then the following statements are logically equivalent.
\begin{enumerate}
\item $\calT_* \subset \calF$ and for any $\calC \subseteq \calT$, $\nu(\bigcup \calC) = \sup_{C \in \calC} \nu(C)$ and $\nu(\bigcap \calC) = \inf_{C \in \calC} \nu(C)$.
\item There exists a chain $\calU \subset \calF$ such that $\calT \subseteq \calU$ and $\nu(\calU) := \{ \nu(A) : A \in \calU \}$ is dense in $[0,1]$.
\item For every $\epsilon > 0$ there exists $N_{\epsilon} \in \N$ such that $| \nu_N(A) - \nu(A) | < \epsilon$ for all $A \in \calT$ and all $N > N_{\epsilon}$. 
\end{enumerate}
Moreover, if any of the three statements holds then 
\begin{enumerate}
\item $\calU_* \subset \calF$ and for any $\calC \subseteq \calU_*$, $\nu(\bigcup \calC) = \sup_{C \in \calC} \nu(C)$ and $\nu(\bigcap \calC) = \inf_{C \in \calC} \nu(C)$, 
\item $\nu(\calU_*) = [0,1]$, 
\item For every $\epsilon > 0$ there exists $N_{\epsilon} \in \N$ such that $| \nu_N(A) - \nu(A) | < \epsilon$ for all $A \in \calU_*$ and all $N > N_{\epsilon}$. 
\end{enumerate}
\end{theorem}
\begin{proof}
($1 \implies 2$) Set $\calT_0 = \calT$. If there is no open interval $(b,c) \subset [0,1]$ of width at least $2^{-1}$ such that $(b,c) \cap \nu(\calT_0) = \emptyset$, then set $\calT_1 := \calT_0$. If there is exactly one such open interval, define sets $B := \bigcup \{ A \in \calT_0 : \nu(A) < b \}$ and  $C := \bigcap \{ A \in \calT_0 : \nu(A) > c \}$. Statement~1 implies $B, C \in \calT_* \subset \calF$, $\nu(B) \leq a < b \leq \nu(C)$ and $(\nu(B),\nu(C)) \cap \nu(\calT_0) = \emptyset$. Form a set $A$ (called a {\em midpoint set}) containing $B$ and every second element of the sequence generated by listing the elements of $C \setminus B$ in increasing order. Then $B \subset A \subset C$ and it is straightforward to show $A \in \calF$ with $\nu(A) = (\nu(B) + \nu(C))/2$. Set $\calT_1$ to be $\calT_0$ plus the midpoint set thus formed. If there are two disjoint open intervals of width at least $2^{-1}$, both of which have empty intersection with $\nu(\calT_0)$, then find the midpoint sets for both intervals and add them to $\calT_0$ to form $\calT_1$. Note there cannot be more than two such intervals. Then $\calT_1 \subset \calF$ is a chain that satisfies $\calT_1^* \subset \calF$ and for any $\calC \subseteq \calT_1$, $\nu(\bigcup \calC) = \sup_{C \in \calC} \nu(C)$ and $\nu(\bigcap \calC) = \inf_{C \in \calC} \nu(C)$. Moreover, $\nu(\calT_1) \subseteq [0,1]$ does not exclude any open intervals in $[0,1]$ of width at least $2^{-1}$. Proceeding inductively, one can generate a non-decreasing sequence of chains $\calT_1 \subseteq \calT_2 \subseteq \ldots$ such that $\nu(\calT_k)$ does not exclude any open intervals in $[0,1]$ of width at least $2^{-k}$. (Note $\nu(\calT_{k-1})$ cannot exclude more than $2^k$ disjoint open intervals in $[0,1]$ of width at least $2^{-k}$, so at most $2^k$ mid-point sets are added to $\calT_{k-1}$ to form $\calT_k$.) Thus the chain $\calU := \bigcup_{k=1}^{\infty} \calT_k$ contains $\calT$ and $\nu(\calU)$ is dense in $[0,1]$.

($2 \implies 3$) Suppose without loss of generality that $\emptyset, \N \in \calU$ (if not, simply add them). Fix $\epsilon > 0$. Then there exists finite $F \subseteq \nu(\calU)$ such that for every $x \in [0,1]$ there are $b, c \in F$ with $b \leq x \leq c$ and $c - b < \epsilon/2$, since $\nu(\calU)$ us dense in $[0,1]$. For each $b \in F$, there exists $A_b \in \calU$ such that $\nu(A_b) = b$. 

Since $F$ is finite, there exists $N_{\epsilon} \in \N$ such that $| \nu_N(A_b) - \nu(A_b) | < \epsilon/2$ for all $N > N_{\epsilon}$ and all $b \in F$. Now for any $A \in \calU$, there exist $b, c \in F$ with $A_b \subseteq A \subseteq A_c$, $b \leq \nu(A) \leq c$ and $c - b < \epsilon/2$. Thus,
\[
| \nu_N(A_b) - \nu(A) | \leq | \nu_N(A_b) - b | + | \nu(A) - b | \leq | \nu_N(A_b) - \nu(A_b) | + | c - b | < \epsilon
\]
and similarly $| \nu_N(A_c) - \nu(A) | < \epsilon$. 

Since $A_b \subseteq A \subseteq A_c$,  
\[
\nu(A) - \epsilon < \nu_N(A_b) \leq \nu_N(A) \leq \nu_N(A_c) < \nu(A) + \epsilon,
\] 
which implies $| \nu_N(A) - \nu(A) | < \epsilon$. Hence Condition~3 holds for all $A \in \calU$, and thus for all $A \in \calT$.

($3 \implies 1$) Consider $A \in \calT_{\cup}$ and fix $\epsilon > 0$. Define $\nu_{\cup}(A) := \sup \{ \nu(C) : C \in \calT, C \subseteq A \}$. Then there exists $C_1 \in \calT$ such that $C_1 \subseteq A$ and $| \nu(C) - \nu_{\cup}(A) | < \epsilon/2$ for all $C \in \calT$ such that $C_1 \subseteq C \subseteq A$. By assumption, there exists $N_{\epsilon} \in \N$ such that $| \nu_N(C) - \nu(C) | < \epsilon/2$ for all $C \in \calT$ and for all $N > N_{\epsilon}$. For any $N > N_{\epsilon}$, there exists $C_2 \in \calT$ such that $C_1 \subseteq C_2 \subseteq A$ and $I_A(n) = I_{C_2}(n)$ for $n = 1, \ldots, N$, so that $| \nu_N(A) - \nu_N(C_2) | = 0$. Thus
\[
| \nu_N(A) - \nu_{\cup}(A) | \leq | \nu_N(A) - \nu_N(C_2) | + | \nu_N(C_2) - \nu(C_2) | + | \nu(C_2) - \nu_{\cup}(A) | < \epsilon.
\]
Hence $\nu_N(A) \rightarrow \nu_{\cup}(A)$, implying $A \in \calF$ with $\nu(A) = \nu_{\cup}(A)$. Similarly, for all $A \in \calT_{\cap}$, $A \in \calF$ with $\nu(A) = \nu_{\cap}(A) :=  \inf \{ \nu(C) : C \in \calT, A \subseteq C \}$. Hence $\calT_{*} = \calT_{\cup} \cup \calT_{\cap} \subset \calF$. 

Now consider $\calC \subseteq \calT$ and define $A := \bigcup \calC$. Then $A \in \calT_{\cup} \subset \calF$ and $\nu(A) = \nu_{\cup}(A)$ as shown in the preceding paragraph. It is straightforward to check $\nu_{\cup}(A) = \sup_{C \in \calC} \nu(C)$. Similarly, $\nu(\bigcap \calC) = \inf_{C \in \calC} \nu(C)$. 

If any of the three statements hold for $\calT$, then Statement~2 also holds with $\calT$ replaced by $\calU$, since trivially $\calU \subseteq \calU$. Thus $\calU_* \subset \calF$, by Statement~1. But then Statement~2 holds with both $\calT$ and $\calU$ replaced by $\calU_*$. Hence Statements~1 and~3 hold with $\calT$ replaced by $\calU_*$. Finally, $\nu(\calU_*) = [0,1]$, since for any $x \in [0,1]$, the set $A_x := \bigcup \{ A \in \calU : \nu(A) \leq x \} \in \calU_*$ with $\nu(A_x) = x$. 
\qed \end{proof}

For chains that are also sequences, the following corollary holds.

\begin{corollary} \label{uniformconvergencesequence}
Consider pairwise disjoint sets $\{ A_k \}_{k=1}^{\infty}$ in $\calF$. Let $B_k = \cup_{i=1}^k A_i$ for each $k$ and let $B = \cup_{i=1}^{\infty} A_i$. The following conditions are logically equivalent:
\begin{enumerate}
\item $B \in \calF$ and $\nu(B) = \sum_{k=1}^{\infty} \nu(A_k)$.
\item For any $\epsilon > 0$ there exists a positive integer $N_{\epsilon} = N_{\epsilon}(B_1, B_2, \ldots)$ such that 
\[
| \nu_N(B_k) - \nu(B_k) | < \epsilon
\]
for all $N \geq N_{\epsilon}$ and for all $k$. 
\end{enumerate}
Moreover, if either statement holds then $| \nu_N(B) - \nu(B) | < \epsilon$ for all $N \geq N_{\epsilon}$.
\end{corollary}
\begin{proof}
Note $\calT = \{ B_k \}_{k=1}^{\infty}$ is a chain in $\calF$. Note also $\calT_{\cap} = \calT$ and $\calT_{*} = \calT_{\cup} = \calT \cup \{ B \}$. 

($1 \implies2$) Statement 1 gives $\calT_{*} = \calT \cup \{ B \} \subset \calF$. For any $\calC \subseteq \calT$, $\bigcap \calC$ is the smallest element of $\calC$, hence $\nu(\bigcap \calC) = \inf_{C \in \calC} \nu(C)$. If there is a largest element of $\calC$, then $\bigcup \calC$ is that largest element, otherwise $\bigcup \calC = B$. In the case of a largest element, $\nu(\bigcup \calC) = \sup_{C \in \calC} \nu(C)$. In the case $\bigcup \calC = B$, $\nu(\bigcup \calC) = \nu(B) =  \sum_{k=1}^{\infty} \nu(A_k) = \sup_{C \in \calC} \nu(C)$ again. Hence the uniform convergence condition holds on all $\calT_{*}$ by Theorem~\ref{uniformconvergence}.

($2 \implies 1$) Statement 2 is the uniform convergence condition of Theorem~\ref{uniformconvergence} as it applies to $\calT$. Hence $B \in \calT_{*} \subset \calF$ and $\nu(B) = \sup_{k=1}^{\infty} \nu(B_k) = \sum_{k=1}^{\infty} \nu(A_k)$.
\qed \end{proof}

The class of chains described in Theorem \ref{uniformconvergence} has a fourth characterisation in terms of maximal chains, as follows.

\begin{corollary} \label{maximal_chains}
Let $\calT \subset \calF$ be a chain of sets. Then $\calT$ satisfies the equivalent conditions of Theorem~\ref{uniformconvergence} if and only if there exists a {\em maximal} chain $\calU_{**} \subset \calF$ (maximal in the sense that it is not a proper subset of any other chain in $\calP(\N)$) such that $\calT \subseteq \calU_{**}$ and $\nu(\calU_{**}) = [0,1]$.
\end{corollary}

\begin{proof}
$(\implies)$
First define the chain $\calU_*$ as described in Theorem \ref{uniformconvergence}, and assume without loss of generality that $\emptyset, \N \in \calU_*$. For every $k \in \N$, let $B_k := \bigcup \{ A \in \calU_* : k \notin A \}$ and $C_k := \bigcap \{ A \in \calU_* : k \in A \}$. Let $D_k = C_k \setminus B_k$ (it is straightforward to verify $B_k \subset C_k$). Note $C_k$ is the smallest set in $\calU_*$ containing $k$, and $B_k$ is the largest set in $\calU_*$ not containing $k$. These sets have the following properties, which are left to the reader to verify.

\begin{enumerate}
    \item If $k' \in D_k$, then $D_k = D_{k'}$.
    \item If $B_k \subseteq A \subseteq C_k$ for some $A \in \calU_*$, $k \in \N$, then either $A = B_k$ or $A=C_k$.
    \item If $A \in \calU_*$, then for any $k \in \N$ either $A \subseteq B_k$ or $C_k \subseteq A$.
\end{enumerate}


For each $k \in \N$, let $D_k = \{x_{k1}, x_{k2}, \ldots \}$, ordered by increasing magnitude; this set may be finite or infinite. Now let $\calU_{**}$ be $\calU_*$ together with all sets of the form $B_k \cup \{x_{k1}, \ldots, x_{kN}\}$, for any $k \in \N$ and $N \in \N$ (if $D_k$ is finite, restrict $N$ accordingly). Then $\calU_{**}$ is a maximal chain, shown as follows. Let $E, F \in \calU_{**}$. If $E, F \in \calU_{*}$, they are comparable since $\calU_*$ is a chain. If $E \in \calU_*, F \in \calU_{**} \setminus \calU_{*}$, then choose $k \in \N$ so that $F$ is of the form $B_k \cup \{x_{k1}, \ldots, x_{kN}\}$. By $3$ above either $E \subseteq B_k$, in which case $E \subseteq F$, or $C_k \subseteq E$, in which case $F \subseteq E$, so in either case $E$ and $F$ are comparable. If $E,F \in \calU_{**} \setminus \calU_{*}$, then they must be of the form $E = B_k \cup \{x_{k1}, \ldots, x_{kN}\}$ and $F=B_{k'} \cup \{x_{k'1}, \ldots, x_{k'N'}\}$; if $k' \in D_k$ then one must contain the other by 1 above, whereas if $k' \notin D_k$ the result follows by noting that in this case either $C_k \subseteq B_{k'}$ or $C_{k'} \subseteq B_{k}$ (by 2 above). Thus $\calU_{**}$ is a chain. 

Suppose there exists a chain $\calV \subseteq \calP(\N)$ with $\calU_{**} \subseteq \calV$, and let $E \in \calV$. Suppose $C_k \subseteq E$ for all $k \in E$. Then $\cup_{k \in E} C_k \subseteq E$, implying $E = \cup_{k \in E} C_k \in \calU_*$. Alternatively, suppose there exists $k \in E$ such that $E \subseteq C_k$. Then $B_k \subseteq E \subseteq C_k$, since $k \notin B_k$. Either $E \in \{ B_k, C_k \} \subset \calU_*$, or there is a largest $N$ such that $x_{kN} \in E$, in which case $E = B_k \cup \{x_{k1}, \ldots, x_{kN}\} \in \calU_{**}$. Hence $\calU_{**}$ is maximal.


Finally, $\calU_{**} \subseteq \calF$, since $\calU_{*} \subseteq \calF$, and every set in $\calU_{**}$ differs from a set in $\calU_{*}$ by at most a finite (and therefore null) set.

$(\impliedby)$ This is immediate from Statement~2 of Theorem \ref{uniformconvergence}.
\qed \end{proof}

\section{Boolean algebras, quotients and the Monotone Class Theorem}
\label{Boolean_review}

The set $\calF$ can in a certain sense be factored by the null sets $\calN$ to produce a simple structure known as a monotone class, on which the induced function $\nu$ is countably additive. This useful result is  Corollary~\ref{countableadditivity} of Section~\ref{quotient_space} below. The proof involves a technique for manipulating chains in $\calP(\N)$ that is here called {\em null modification}, described in Section~\ref{null_modification_section}. Both  sections involve Boolean quotients, and while the theory of Boolean algebras and their quotients will be familiar to many readers, it may nevertheless be helpful to briefly review key definitions and results. That is the purpose of this section. There are no new results in this section, but it does contain a slight generalisation of the monotone class theorem for Boolean algebras (Theorem~\ref{monotone_for_Boolean}), based on the proof for fields of sets given in Paul Halmos' classic text on Measure Theory.

A {\em Boolean algebra} \cite{givant2009} is an abstraction of a field of sets consisting of a non-empty set $\calA$ equipped with two binary operators called {\em join} $\vee$ and {\em meet} $\wedge$, a unary {\em complement} operator $^{\prime}$ and containing special elements called the {\em zero} $0$ and {\em unit} (or {\em one}) $1$, satisfying the following axioms:
\begin{eqnarray*}
p \wedge 1 = p, & p \vee 0 = p, \\
p \wedge p^{\prime} = 0, & p \vee p^{\prime} = 1, \\
p \wedge q = q \wedge p, & p \vee q = q \vee p, \\
p \wedge (q \vee r) = (p \wedge q) \vee (p \wedge r), & p \vee (q \wedge r) = (p \vee q) \wedge (p \vee r).
\end{eqnarray*} 
These four pairs of axioms are known as the {\em identity laws}, {\em complement laws}, {\em commutative laws} and {\em distributive laws} respectively, and entail a number of other well known identities including associative laws and De Morgan's laws. Other common Boolean operators and relations can be composed from the meet, join and complement, for example $p - q := p \wedge q^{\prime}$ and $p + q := (p \wedge q^{\prime}) \vee (p^{\prime} \wedge q)$.  Another example is the partial order defined by $p \leq q \iff p \vee q = q$.  

The simplest example of a Boolean algebra is the set $\{0, 1\}$, with basic Boolean operations defined by 
\begin{eqnarray*}
0 \wedge 0 = 0, & 1 \wedge 1 = 1, & 0 \wedge 1 = 1 \wedge 0 = 0,\\
0 \vee 0 = 0, & 1 \vee 1 = 1, & 0 \vee 1 = 1 \vee 0 = 1,\\
0^{\prime} = 1, & 1^{\prime} = 0.
\end{eqnarray*}

Any field of sets $\calA \subseteq \calP(X)$ on an arbitrary set $X$ is a Boolean algebra with pairwise intersection $\cap$ as the {\em meet} operator, pairwise union $\cup$ as the {\em join} operator, set complement $^c$ as the Boolean complement operator, the empty set $\emptyset$ as the zero and $X$ as the unit. Note also $p - q$ is the set difference $p \setminus q$, $p + q$ is the symmetric difference $p \triangle q$ and the partial order $p \leq q$ is the subset relation $p \subseteq q$.

A {\em Boolean homomorphism} is a mapping $f : \calA \rightarrow \calA^{\prime}$ between Boolean algebras $\calA$ and $\calA^{\prime}$ that respects the basic set operations. Specifically, a homomorphism satisfies
\begin{align*}
f(A \vee B) &= f(A) \vee f(B), \\
f(A \wedge B) &= f(A) \wedge f(B), \\
f(0) &= 0, \mbox{ and } \\
f(1) &= 1, 
\end{align*}
for all $A, B \in \calA$. It follows that $f(A^{\prime}) = f(A)^{\prime}$, and indeed all finite combinations of basic Boolean operations are respected, including the partial order, that is $p \leq q \implies f(p) \leq f(q)$. A {\em Boolean isomorphism} is a homomorphism with an inverse homomorphism.  

A {\em Boolean ideal} $\calM$~\cite{givant2009} is a non-empty subset of a Boolean algebra $\calA$ satisfying the following axioms: 
\begin{eqnarray*}
p,q \in \calM & \implies & p \vee q \in \calM, \\
p \in \calM, q \in \calA & \implies & p \wedge q \in \calM. 
\end{eqnarray*} 
For example, for any charge space $(X,\calA,\mu)$, the set $\mu^{-1}(0) := \{ A \in \calA : \mu(A) = 0 \}$, called the {\em kernel} of $\mu$, is an ideal of $\calA$, and the collection of {\em null sets} $\{ A \in \calP(X) : \mu^*(A) = 0 \}$, where $\mu^*(A) := \inf \{ \mu(B) : B \in \calA, A \subseteq B \}$, forms an ideal of $\calP(X)$. The set $\calN$ defined in Section~\ref{F_and_nu} is a Boolean ideal of $\calP(\N)$, since by Lemma~\ref{calFproperties}(6), $A \cup B \in \calN$ for all $A,B \in \calN$, and $A \cap C \in \calN$ for all $A \in \calN$ and $C \in \calP(\N)$. 

A Boolean ideal $\calM$ induces an equivalence relation $\sim$ on the containing Boolean algebra $\calA$ such that
\[
p \sim q \iff p+q \in \calM.
\]
The collection of equivalence classes $\calA / \calM := \{ [ p ] : p \in \calA \}$, where $[ p ]$ denotes the equivalence class of $p$ under the equivalence relation induced by $\calM$, is called the {\em quotient} of $\calA$ by $\calM$. When the Boolean algebra in question is ambiguous it is convenient to write $[ p ]_{\calA/\calM}$ to identify both the underlying algebra $\calA$ and the ideal $\calM$.  

A key example in this paper is the Boolean quotient
\[
\calP(\N) / \calN := \{ [A]  : A \in \calP(\N) \},
\] 
where $[A] $ denotes the equivalence class of $A$ under the equivalence relation $A \sim B \iff A \triangle B \in \calN$. 

A quotient is itself a Boolean algebra when equipped with the Boolean operators $[ p ] \wedge [ q ] := [ p \wedge q ]$, $[ p ] \vee [ q ] := [ p \vee q ]$, $[ p ]^{\prime} := [ p^{\prime} ]$, and with $[ 0 ]$ and $[ 1 ]$ as the zero and unit respectively. The map $p \mapsto [p]$ is a Boolean homomorphism. This map respects the partial order, and in fact $[ p ] \leq [ q ]$ if and only if there exists $p^* \in [ p ]$ such that $p^* \leq q$, or equivalently there exists $q^* \in [ q ]$ such that $p \leq q^*$.

For any $\calA^{\prime} \subseteq \calA$ (not necessarily a sub-algebra), define $[\calA^{\prime}] := \{ [A] : A \in \calA^{\prime} \}$. Where possible, parentheses will be omitted when the argument is contained in square brackets. 

If $(X,\calA,\mu)$ is a charge space and $\calM$ is the kernel of $\mu$, the induced function $\mu : \calA / \calM \rightarrow \R$ given by $\mu[ A ] := \mu(A)$ for all $A \in \calA$ is finitely additive.

A new version of the monotone class theorem, which generalises the version in \cite{keisler1977monotone}, is presented below. The new result makes use of the following definitions, some of which are non-standard. 

A Boolean algebra $\calA$ is said to be {\it countably complete} if every countable subset $\{ p_k \}_{k=1}^{\infty}$ in $\calA$ has a least upper bound in $\calA$. A subalgebra $\calB \subseteq \calA$ will here be called {\it countably complete} if every countable subset $\{ p_k \}_{k=1}^{\infty}$ in $\calB$ has an upper bound in $\calB$ that is less than any other upper bound of this subset in $\calA$. This upper bound is called the {\em supremum} of the subset and denoted $\vee_{k=1}^{\infty} p_k$. In that case, it is straightforward to show (by taking complements) that every countable subset also has a lower bound in $\calB$ that is greater than any other lower bound of the subset in $\calA$, called the {\em infimum} of the subset and denoted $\wedge_{k=1}^{\infty} p_k$. By definition these two elements are unique. 

An important subtlety is that a proper subalgebra $\calB \subset \calA$ will not here be called countably complete if it is only true that every countable subset of $\calB$ has an upper bound in $\calB$ that is less than any other upper bound of that subset in $\calB$: it must be less than any other upper bound of that subset in $\calA$. The reason for this requirement is that, without it, the supremum of a countable subset $\calC$ of a subalgebra $\calB_1$ could differ from the supremum of $\calC$ when viewed as a subset of a distinct subalgebra $\calB_2$. This can occur even if $\calB_1$ and $\calB_2$ are both countably complete algebras when the containing algebra $\calA$ is ignored. Thus the requirement is needed to ensure the supremum of $\calC$ is uniquely defined across all countably complete subalgebras of $\calA$. 

A subset $\calM$ of $\calA$ will be called a {\it monotone class} if: 
\begin{enumerate}
\item for any non-decreasing sequence $p_1 \leq p_2 \leq \ldots$ in $\calM$, there is an upper bound in $\calM$ that is less than any other upper bound of the sequence in $\calA$, and
\item for any non-increasing sequence $p_1 \geq p_2 \geq \ldots$ in $\calM$, there is a lower bound in $\calM$ that is greater than any other lower bound of the sequence in $\calA$.
\end{enumerate}
Similarly to countably complete subalgebras, this least upper bound will be called the supremum of the sequence, denoted $\vee_{k=1}^{\infty} p_k$, and this greatest lower bound will be called the infimum of the sequence, denoted $\wedge_{k=1}^{\infty} p_k$. The same word of caution is necessary here as for countably complete subalgebras: $\vee_{k=1}^{\infty} p_k$ must be less than any other upper bound of the sequence in $\calA$, not just in $\calM$, and similarly $\wedge_{k=1}^{\infty} p_k$ must be greater than any other lower bound in $\calA$, not just in $\calM$.

In fact, the version of the monotone class theorem presented in \cite{keisler1977monotone} also requires countably complete subalgebras and monotone classes to be understood in this sense, though this is not explicitly stated. Note the version of the monotone class theorem presented in that paper differs from the one below in requiring the containing algebra $\calA$ to be countably complete.

The notation $\bigvee \calC$ will be used below to denote the supremum of a countable subset or non-decreasing sequence $\calC$. Similarly, $\bigwedge \calC$ denotes the infimum of a countable subset or non-increasing sequence $\calC$.

Proof of the monotone class theorem depends on the following lemma, which is analogous to \cite[Thm. A, p. 27]{halmos2013measure}.

\begin{lemma} \label{BsubA_is_monotone}
Suppose $\calA$ is a Boolean algebra, and $\calM$ is a subalgebra that is also a monotone class. Then $\calM$ is a countably complete subalgebra.
\end{lemma}

\begin{proof}
Let $\calC=\{p_1,p_2, \ldots \}$. Then, since $\calM$ is a Boolean algebra, the elements $p_1, p_1 \vee p_2, p_1 \vee p_2 \vee p_3, \ldots$ are also in $\calM$. These elements form an increasing sequence, hence this sequence has an upper bound in $\calM$ that is less than any other upper bound in $\calA$. It may be checked this upper bound is also the supremum of $\calC$, implying $\calM$ is countably complete. 
\qed \end{proof}


As stated above, the version of the monotone class theorem below is adapted from \cite{keisler1977monotone}. There it is claimed that the result is proved in \cite{halmos2013measure}; however this may be an example of mathematical folklore, as the result in that reference applies only to fields of sets, which are less general than Boolean algebras. The following proof is derived from the proof of \cite[Thm. B, p. 27]{halmos2013measure}.

\begin{theorem} \label{monotone_for_Boolean}
Let $\calA$ be a Boolean algebra and let $\calM \subseteq \calA$ be a monotone class. Let $\calA_0 \subseteq \calM$ be a subalgebra of $\calA$, and define $\sigma(\calA_0) \subseteq \calM$ to be the smallest monotone class in $\calA$ that contains $\calA_0$. Then $\sigma(\calA_0)$ is also the smallest countably complete subalgebra of $\calA$ that contains $\calA_0$.
\end{theorem}

\begin{proof}
It will be sufficient to show that $\sigma := \sigma(\calA_0)$ is a Boolean subalgebra, for then it will be countably complete by Lemma~\ref{BsubA_is_monotone}, and in fact it will be the smallest countably complete subalgebra containing $\calA_0$ because any smaller countably complete subalgebra containing $\calA_0$ would also be a smaller monotone class containing $\calA_0$. 

For $q \in \calA$, let
$$ 
K(q) = \{p \in \calA: p-q, q-p, q \vee p \in \sigma \}.
$$
These sets possess a convenient symmetry: $p \in K(q)$ if, and only if, $q \in K(p)$. Suppose $p_1 \leq p_2 \leq \ldots$ is a non-decreasing sequence of elements in $K(q)$. Then $p_k \vee q \in \sigma$ for all $k$, and since $\sigma$ is a monotone class it follows that 
$$
\vee_{k=1}^{\infty} (p_k \vee q) = \Big(\vee_{k=1}^{\infty} p_k\Big) \vee q \in \sigma.
$$

Similar arguments show that $\vee_{k=1}^{\infty} p_k - q, q-\vee_{k=1}^{\infty} p_k \in \sigma$ as well, and it follows that $\vee_{k=1}^{\infty} p_k \in K(q)$. A parallel argument shows that $\wedge_{k=1}^{\infty} p_k \in K(q)$, and thus $K(q)$ is a monotone class. If $q \in \calA_0$, then $\calA_0 \subseteq K(q)$, and thus, since $K(q)$ is a monotone class, $\sigma \subseteq K(q)$. However, the symmetry mentioned above now implies that if $p \in \sigma$ then $q \in K(p)$ for any $q \in \calA_0$, and then, since $K(p)$ is a monotone class, that $\sigma \subseteq K(p)$. This implies in particular that, for any $p,q \in \sigma$, the elements $p-q, q-p, q \vee p$ are all in $\sigma$, and hence $q \wedge p \in \sigma$ as well, since $q \wedge p = (q \vee p) - (p-q) - (q-p)$. It follows that $\sigma$ is a Boolean subalgebra.
\qed \end{proof}

\section{Null modification}
\label{null_modification_section}

This section develops another analytic tool for studying chains in $\calP(\N)$: a construction that is here called a {\em null modification}. A null modification takes a set in $\calP(\N)$ and constructs a new set of a form described in Property~9 of Lemma~\ref{calFproperties}, using Algorithm~1.

\begin{lemma} \label{nullmodification}
For any $A \in \calP(\N)$, Algorithm~\ref{algo:nullmod} decomposes $A$ into disjoint sets $A^{\prime} \in \calP(\N)$ and $F \in \calN$ such that
\begin{enumerate}
\item $A = A^{\prime} \cup F$,
\item $\nu^+(A^{\prime}) = \nu^+(A)$, and
\item $\nu_N(A^{\prime}) \leq \nu^+(A^{\prime})$ for all $N \in \calN$.
\end{enumerate}
Moreover, if $A \in \calF$, then $A^{\prime} \in \calF$.
\end{lemma}

\begin{algorithm}[h!]
\caption{{\bf Null modification:} Given $A \in \calP(\N)$, construct $A^{\prime} \in \calP(\N)$ and $F \in \calN$}
\label{algo:nullmod}
  \begin{algorithmic}
    \STATE{Set $A^{\prime} = F = \emptyset$.}
    \FOR{$N=1, 2, \ldots$}
      \IF{$N \in A$}
	\STATE{Add $N$ to $A^{\prime}$}
      	\IF{$\nu_N(A^{\prime}) > \nu^+(A)$}
	  \STATE{Remove $N$ from $A^{\prime}$ and add it to $F$.}
	\ENDIF
      \ENDIF
    \ENDFOR
  \end{algorithmic}
\end{algorithm}

\begin{proof}
Algorithm~\ref{algo:nullmod} trivially ensures $A^{\prime}$ and $F$ are disjoint, $A = A^{\prime} \cup F$ and $\nu_N(A^{\prime}) \leq \nu^+(A)$ for all $N$. Moreover, if $N \in F$, then 
\[
\nu_N(A^{\prime}) > \nu^+(A) - \frac{1}{N}.
\]
Next show $F \in \calN$ as follows. This is trivial if $F$ is a finite set, so suppose it is infinite. Fix $\epsilon > 0$ and choose $N_{\epsilon}$ so that the following conditions are met:
\begin{enumerate}
\item $N_{\epsilon} \in F$,
\item $N_{\epsilon} > 2/ \epsilon$, and
\item $\nu_N(A) < \nu^+(A) + \epsilon /2$ for all $N \geq N_{\epsilon}$.
\end{enumerate}
Now for any $N \geq N_{\epsilon}$, if $N \in F$ then
\begin{eqnarray*}
\nu_N(F) & = & \nu_N(A) - \nu_N(A^{\prime}) \\
& < & \nu^+(A) + \frac{\epsilon}{2} - \left( \nu^+(A) - \frac{1}{N} \right) \\
& \leq & \frac{\epsilon}{2} + \frac{1}{N_{\epsilon}} \\
& < & \epsilon.
\end{eqnarray*}
If $N \notin F$ then
\[
\nu_N(F) < \nu_{N^{\prime}}(F) < \epsilon
\]
where $N^{\prime}$ is the largest integer less than $N$ for which $N^{\prime} \in F$, noting that $N^{\prime} \geq N_{\epsilon}$. Hence $\lim_{N \rightarrow \infty} \nu_N(F) = 0$, implying $F \in \calN$. Lemma~\ref{calFproperties}(6) gives $\nu^+(A^{\prime}) = \nu^+(A)$. If $A \in \calF$, Lemma~\ref{calFproperties}(6) gives $A^{\prime} \in \calF$ with $\nu(A^{\prime}) = \nu(A)$.  
\qed \end{proof}

Recalling the example given at the end of Section \ref{F_and_nu}, it can be checked that the set $\calO$ defined there, the set of all odd numbers at least 3, can be obtained by applying this algorithm to the set of all odd numbers. The null set removed by the algorithm is simply $\{1\}$.

Null modification can be used to transform chains in $\calF$ to acquire a useful topological property, defined in terms of the following pseudo-metric. Let $d_{\nu}(B,C) := \nu^+(B \triangle C)$ for all $B, C \in \calP(\N)$. Trivially, $d_{\nu}(B,B)=0$ and $d_{\nu}(B,C)=d_{\nu}(C,B)$. The triangle inequality $d_{\nu}(A,B) + d_{\nu}(B,C) \geq d_{\nu}(A,C)$ follows from the fact that
\[
\begin{split}
    A \triangle C & = (A \setminus C) \cup (C \setminus A) \subseteq (A \setminus B \cup B \setminus C) \cup (C \setminus B \cup B \setminus A) \\
    & = (A \setminus B \cup B \setminus A) \cup (C \setminus B \cup B \setminus C) = (A \triangle B) \cup (B \triangle C),
\end{split}
\]
so that $\nu_N(A \triangle C) \leq \nu_N(A \triangle B) + \nu_N(B \triangle C)$. 


This pseudo-metric is related to the continuity of the set functions $\nu^+$, $\nu^-$ and $\nu$ on chains in $\calP(\N)$ or $\calF$ in the following sense.

\begin{lemma} \label{sup_and_inf}
Consider a chain $\calS \subseteq \calP(\N)$. Then
\begin{enumerate}
\item if $\inf \{ d_{\nu}(\bigcup \calS, A) : A \in \calS \} = 0$, then $\nu^+(\bigcup \calS) = \sup \{ \nu^+(A) : A \in \calS \}$, and
\item if $\inf \{ d_{\nu}(A, \bigcap \calS) : A \in \calS \} = 0$, then $\nu^-(\bigcap \calS) = \inf \{ \nu^-(A) : A \in \calS \}$.
\end{enumerate}
Moreover, if $\calS \subseteq \calF$ and $\bigcup \calS \in \calF$, the converse of the first result holds, and if $\calS \subseteq \calF$ and $\bigcap \calS \in \calF$, the converse of the second result holds.
\qed \end{lemma}

\begin{proof}
For 1, first note $\sup \{ \nu^+(A) : A \in \calS \} \leq \nu^+(\bigcup \calS)$, since $\nu^+(A) \leq \nu^+(\bigcup \calS)$ for all $A \in \calS$. To show the reverse inequality, fix $\epsilon > 0$ and choose $A \in \calS$ such that $d_{\nu}(\bigcup \calS, A) < \epsilon$. By Lemma~\ref{calFproperties}(5), $d_{\nu}(\bigcup \calS, A) \geq \nu^+(\bigcup \calS) - \nu^+(A)$, hence
\[
\nu^+(\bigcup \calS) < \nu^+(A) + \epsilon \leq \sup \{ \nu^+(A) : A \in \calS \} + \epsilon.
\]
Let $\epsilon \rightarrow 0$ to obtain the first result. Then 2 follows by taking complements.

If $\calS \subseteq \calF$ and $\bigcup \calS \in \calF$, then $\bigcup \calS \setminus A \in \calF$ with $\nu(\bigcup \calS \setminus A) = \nu(\bigcup \calS) - \nu(A)$ for all $A \in \calS$, by Lemma~\ref{calFproperties}(5). Taking infima gives 
\[
\inf \{ d_{\nu}(\bigcup \calS, A) : A \in \calS \} = \nu(\bigcup \calS) - \sup \{ \nu(A) : A \in \calS \},
\]
implying the converse of 1. The converse of 2 follows by taking complements.
\qed \end{proof}

The main result in this section (Theorem~\ref{nullmodification_thm}) establishes that any chain $\calT$ in $\calF$ can be mapped to a chain satisfying the conditions of Lemma~\ref{sup_and_inf} on all subchains $\calS \subseteq \calT$. The construction involves first modifying the chain so that the conditions of Lemma~\ref{sup_and_inf}(1) hold for all subsets of the chain. 

To describe this construction, it will be convenient to introduce the following notation. Define
\begin{eqnarray*}
\calC(A,B] & := & \{ C \in \calC : A \subset C \subseteq B \} 
\end{eqnarray*}
to represent {\em sub-intervals} of a chain $\calC \subset \calP(N)$. Here $A, B \in \calP(\N)$ but are not necessarily elements of $\calC$. 

It will also be convenient to define the {\em left end-points} of a chain $\calC \subset \calF$ to be sets of the form $\bigcap \calS$, where $\calS \subset \calC$, such that:
\begin{enumerate}
\item $\bigcap \calS \in \calF$, and
\item $\exists \epsilon > 0$ such that for any $A \in \calC$,  $d_{\nu}(A,\bigcap \calS) < \epsilon \implies A \supseteq \bigcap \calS$.
\end{enumerate}
In other words, there is a ``gap'' of width at least $\epsilon$ to the left of $\nu(\bigcap \calS)$ in $\nu(\calC)$. A chain $\calC$ can have at most countably many left endpoints because there can be at most countably many disjoint sub-intervals in the interval $[0,1]$, corresponding to these gaps.

The subscripted $\cup$ in Property~4 of the following lemma represents closure under unions, reprising the notation introduced in the paragraph before Lemma~\ref{basicfacts}.

\begin{lemma} \label{nullmodification_oneside}
Consider a countable chain $\calC \subset \calF$ such that $\nu(A) \neq \nu(B)$ for distinct $A, B \in \calC$. There exists a map $\psi : \calC \rightarrow \calF$ such that:
\begin{enumerate}
\item for all $A \in \calC$, $\psi(A) \subseteq A$ with $A \setminus \psi(A) \in \calN$, 
\item for all $A, B \in \calC$, $A \subseteq B \implies \psi(A) \subseteq \psi(B)$, 
\item for all $\calS \subseteq \calC$, 
\begin{enumerate}
\item $\bigcup \psi(\calS) \in \calF$ with $\nu(\bigcup \psi(\calS)) = \sup \{ \nu(A) : A \in \psi(\calS) \}$, 
\item if $\bigcap \calS \in \calF$ with $\nu(\bigcap \calS) = \inf \{ \nu(A) : A \in \calS \}$, then $\bigcap \psi(\calS) \in \calF$ with $\nu(\bigcap \psi(\calS)) = \inf \{ \nu(A) : A \in \psi(\calS) \}$, 
\end{enumerate}
\item for all $A \in \psi(\calC)_{\cup}$ and $N \in \N$, $\nu_N(A) \leq \nu(A)$, and
\item if there exists a second map $\rho : \calC \rightarrow \calF$ such that $\rho(A) \triangle A \in \calN$ for all $A \in \calC$, then for any $\calS \subseteq \calC$ with $\bigcup \rho(\calS) \in \calF$ and $\nu(\bigcup \rho(\calS)) = \sup \{ \nu(A) : A \in \rho(\calS) \}$,
\[
\left( \bigcup \psi(\calS) \right) \triangle \left( \bigcup \rho(\calS) \right) \in \calN.
\] 
\end{enumerate}
\end{lemma}

\begin{proof}
Without loss of generality, suppose $\calC$ contains its left end-points. Note the chain remains countable if its left end-points are added, and also retains the property $\nu(A) \neq \nu(B)$ for distinct $A, B \in \calC$. To see the latter claim, note that if a left end-point has the same charge as some $A \in \calC$, then in fact $A$ is already a left end-point of $\calC$, and no new set need be added. The added end-points and their images under $\psi$ can be discarded at the end of the construction, and the stated properties will be retained.

Arbitrarily order the elements of $\calC$ as a sequence $\{ C_k \}_{k=1}^{\infty}$. (This assumes $\calC$ is infinite, but the proof that follows also applies if $\calC$ is finite, with minimal modification.) For each $k \in \N$, define $B_k$ to be the largest set in the sequence $C_1, \ldots, C_{k-1}$ that is a proper subset of $C_k$, if such a set exists (that is, $B_k := \bigcup \{ C_j : j < k, C_j \subset C_k  \}$). Otherwise, set $B_k := \emptyset$. 

Set $\psi_0(C) = C$ for all $C \in \calC$, and assume inductively that $\psi_{k-1}(\calC) \subset \calF$. This is trivially true for $k=1$, since $\calC \subset \calF$. Sequentially define $\psi_k$ for each $k \in \N$ as follows:
\begin{enumerate}
\item Apply Algorithm~\ref{algo:nullmod} to decompose $A_k := \psi_{k-1}(C_k) \setminus \psi_{k-1}(B_k)$ into disjoint sets $A_k ^{\prime} \in \calF$ and $F_k \in \calN$. 
\item For each $C \in \calC$, define
\[
\psi_k(C) := \left\{
\begin{array}{ll}
\psi_{k-1}(C) \setminus F_k & \mbox{ if } C \in \calC(B_k, C_k] \\
\psi_{k-1}(C) & \mbox{ otherwise.}
\end{array}
\right.
\]
\end{enumerate}
Note $\psi_k(C)$ differs from $C$ by the removal of at most $k$ null sets for each $k \in \N$ and $C \in \calC$. Hence $\psi_k(\calC) \subset \calF$.

For each $k \in \N$, define $\psi(C_k) := \psi_k(C_k)$. Then $\psi(C_k) \in \calF$ and Property~1 holds.

Also note that for all $A, B \in \calC$ and $k \in \N$, $A \subseteq B \implies \psi_k(A) \subseteq \psi_k(B)$. Moreover, $\psi_j(C_k) = \psi_k(C_k)$ for all $j > k$. Hence Property~2 holds because if $C_j \subseteq C_k$, then 
\[
\psi(C_j) = \psi_{\max\{j,k\}} (C_j) \subseteq \psi_{\max\{j,k\}} (C_k) = \psi(C_k).
\]

Next note that for all $k \in \N$,
\[
\psi(C_k) = \bigcup \{ A_j^{\prime} : j \leq k, C_j \subseteq C_k \},
\] 
where the components of the union are disjoint. This is shown by induction. It is trivially true for $k = 1$, since $\psi(C_1) = A_1^{\prime}$. Given it is true for all $j < k$, then since $B_k \in \{ \emptyset, C_1, \ldots, C_{k-1} \}$, 
\[
\psi_{k-1}(B_k) = \bigcup \{ A_j^{\prime} : j \leq k-1, C_j \subseteq B_k \},
\]
where the union is disjoint. Moreover, 
\begin{eqnarray*}
\psi(C_k) &=& \psi_k(C_k) = \psi_{k-1}(C_k) \setminus F_k = \psi_{k-1}(B_k) \cup A_k^{\prime} \\ 
&=& \bigcup \{ A_j^{\prime} : j \leq k, C_j \subseteq C_k \}
\end{eqnarray*}
with $A_k^{\prime}$ disjoint from $\psi_{k-1}(B_k)$.

For all $k, N \in \N$, $\nu_N(A_k^{\prime}) \leq \nu(A_k^{\prime})$ by Lemma~\ref{nullmodification}. Hence 
\[
\nu_N(\psi(C_k)) = \sum_{j \leq k: C_j \subseteq C_k} \nu_N(A_j^{\prime}) \leq \sum_{j \leq k: C_j \subseteq C_k} \nu(A_j^{\prime}) = \nu(\psi(C_k))
\]
by Lemma~\ref{calFproperties}(7). Property~3a thus follows by Lemma~\ref{calFproperties}(9).

To show 3b, first suppose $\bigcap \calS$ is a left end-point of $\calC$. Then $\bigcap \calS \in \calC$. Also note $\psi(\bigcap \calS) \subseteq \bigcap \psi(\calS)$, since $\psi(\bigcap \calS) \subseteq \psi(A)$ for all $A \in \calS$. Hence 
\begin{eqnarray*}
\nu(\bigcap \calS) &=& \nu(\psi(\bigcap \calS)) \\
&\leq& \nu^-(\bigcap\psi(\calS)) \\
&\leq& \nu^+(\bigcap\psi(\calS)) \\
&\leq& \inf \{ \nu(\psi(A)) : A \in \calS \} \\
&=& \inf \{ \nu(A) : A \in \calS \},
\end{eqnarray*}
implying $\bigcap\psi(\calS) \in \calF$ with $\nu(\bigcap\psi(\calS)) = \inf \{ \nu(A) : A \in \psi(\calS) \}$. On the other hand, if $\bigcap \calS$ is not a left end-point of $\calC$, then for any $\epsilon > 0$ there is $A \in \calC$ with $A \subset \bigcap \calS$ and $\nu(A) > \nu(\bigcap\calS) - \epsilon$. Hence 
\begin{eqnarray*}
\nu(\bigcap \calS) - \epsilon &<& \nu(A) \\
&=& \nu(\psi(A)) \\
&\leq& \nu^-(\bigcap\psi(\calS)) \\
&\leq& \nu^+(\bigcap\psi(\calS)) \\
&\leq& \inf \{ \nu(\psi(A)) : A \in \calS \} \\
&=& \inf \{ \nu(A) : A \in \calS \},
\end{eqnarray*}
again implying $\bigcap\psi(\calS) \in \calF$ with $\nu(\bigcap\psi(\calS)) = \inf \{ \nu(A) : A \in \psi(\calS) \}$. 

For any $\calS \subseteq \calC$,
\[
A := \bigcup \{ \psi(C) : C \in \calS \} = \cup_{i=1}^{\infty} \psi(C_{k_i})
\] 
where $k_1$ is the smallest positive integer $k$ such that $\psi(C_k) \subseteq A$, $k_2$ is the smallest positive integer $k$ such that $\psi(C_1) \subset \psi(C_k) \subseteq A$ and so on. This sequence is infinite for $A \not\in \calC$. It follows that $A = \cup_{j=1}^{\infty} A_{k_j}^{\prime}$,
where the union is disjoint, since for each $i \in \N$, $\psi(C_{k_i}) = \cup_{j=1}^{i} A_{k_j}^{\prime}$ is a disjoint union. Hence Property~4 holds because 
\[
\nu_N(A) = \sum_{j=1}^{\infty} \nu_N(A_{k_j}^{\prime}) \leq \sum_{j=1}^{\infty} \nu(A_{k_j}^{\prime}) \leq \nu(A)
\]
by Lemma~\ref{calFproperties}(8). 

For 5, define a third map $\tau(A) := \psi(A) \cap \rho(A)$ for all $A \in \calC$. Then since $\psi(A) \triangle A \in \calN$ and $\rho(A) \triangle A \in \calN$, one must have $\psi(A) \triangle \rho(A) \in \calN$ and then also $\tau(A) \triangle \psi(A) \in \calN$, using Lemma~\ref{calFproperties}(6). The same lemma also gives $\rho(A) \in \calF$ and $\tau(A) \in \calF$ for $A \in \calC$. Now, for any $\calS \subseteq \calC$,
\begin{eqnarray*}
\nu^+ \left( \bigcup \psi(\calS) \setminus \bigcup \tau(\calS) \right) & \leq & \inf \left\{ \nu \left( \bigcup \psi(\calS) \setminus \tau(A) \right) : A \in \calS \right\} \\
& = & \nu(\bigcup \psi(\calS)) - \sup \{ \nu(\tau(A)) : A \in \calS \} \\
& = & \nu(\bigcup \psi(\calS)) - \sup \{ \nu(\psi(A)) : A \in \calS \} \\
& = & 0,
\end{eqnarray*}
using 3a. Noting $\bigcup \tau(\calS) \subseteq \bigcup \psi(\calS)$, this implies 
\[
\left(\bigcup \psi(\calS)\right) \triangle \left(\bigcup \tau(\calS)\right) \in \calN.
\]
If $\bigcup \rho(\calS) \in \calF$ with $\nu(\bigcup \rho(\calS)) = \sup \{ \nu(A) : A \in \rho(\calS) \}$, then one may apply a similar argument to $\rho$ instead of $\psi$, giving $\left(\bigcup \rho(\calS)\right) \triangle \left(\bigcup \tau(\calS)\right) \in \calN$, and hence 
\[
\left( \bigcup \psi(\calS) \right) \triangle \left( \bigcup \rho(\calS) \right) \in \calN.
\] 
\qed \end{proof}

The above lemma provides a way of transforming any countable, pairwise disjoint collection of sets in $\calF$ into a similar collection on which $\nu$ is countably additive.

\begin{corollary} \label{nullmodification_sequence}
Consider pairwise disjoint sets $\{ A_k \}_{k=1}^{\infty}$ in $\calF$. For each $k$, there exists $A_k^{\prime} \in \calF$ such that
\begin{enumerate}
\item $A_k^{\prime} \subseteq A_k$ with $F_k := A_k \setminus A_k^{\prime} \in \calN$,  
\item $A := \cup_{k=j}^{\infty} A_k^{\prime} \in \calF$ with $\nu(A) = \sum_{k=1}^{\infty} \nu(A_k^{\prime})$,
\item $\nu_N(A_k^{\prime}) \leq \nu(A_k^{\prime})$ for all $k,N \in \calN$, and $\nu_N(A) \leq \nu^+(A)$ for all $N \in \calN$,
\item if there exists $\{ A_k^{\prime\prime} \}_{k=1}^{\infty} \subset \calF$ such that 
\begin{enumerate}
\item $A_k \triangle A_k^{\prime\prime} \in \calN$ for each $k \in \N$, and
\item $\cup_{k=j}^{\infty} A_k^{\prime\prime} \in \calF$ with $\nu(\cup_{k=j}^{\infty} A_k^{\prime\prime}) = \sum_{k=1}^{\infty} \nu(A_k^{\prime\prime})$,
\end{enumerate}
then $\left( \cup_{k=1}^{\infty} A_k^{\prime} \right) \triangle  \left( \cup_{k=1}^{\infty} A_k^{\prime\prime} \right) \in \calN$.
\end{enumerate}
\end{corollary}

\begin{proof}
First consider the case $\nu(A_k) > 0$ for all $k \in \N$. Then one may apply Lemma~\ref{nullmodification_oneside} to the chain $\calC := \{ B_j \}_{j=1}^{\infty}$, where $B_j := \cup_{k=1}^j A_k$ for $j \in \N$, $A_1^{\prime} = \psi(B_1)$ and $A_{k+1}^{\prime} := \psi(B_{k+1}) \setminus \psi(B_k)$ for $k \in \N$. 

Property~1 follows because the construction in the proof of Lemma~\ref{nullmodification_oneside} uses Algorithm~\ref{algo:nullmod} to remove the null set $F_k$ from $A_k$ to produce $A_k^{\prime}$. Property~2 follows from Lemma~\ref{nullmodification_oneside}(3a) with $\calS = \calC$. Property~3 is immediate from Lemma~\ref{nullmodification_oneside}(4). Property~4 follows from Lemma~\ref{nullmodification_oneside}(5), with $\rho(B_j) := \cup_{k=1}^j A_k^{\prime\prime}$ for $j \in \N$. 

Now consider the case $\nu(A_k) = 0$ for some $k \in  \N$. For each such $k$, set $A_k^{\prime} := \emptyset$ and apply Lemma~\ref{nullmodification_oneside} as above to the remaining elements of $\{ A_k \}_{k=1}^{\infty}$ (ie. those with non-zero charge). Then Properties~1 and the first part of~3 hold for $A_k^{\prime} = \emptyset$, as for the sets with non-zero charge. Property~2 and the latter part of~3 are properties of $A := \cup_{k=j}^{\infty} A_k^{\prime}$, to which sets with $A_k^{\prime} = \emptyset$ make no contribution. Property~4 is also a property of $A$, and will therefore hold provided
\[
\nu \bigg( \bigcup \{ A_k^{\prime\prime} : \nu(A_k) = 0 \} \bigg) = 0.
\]
This must be the case, otherwise condition~4b would not hold.
\qed \end{proof}

One can now apply Lemma~\ref{nullmodification_oneside} twice - to a given chain and to the corresponding chain of complements - to obtain a chain that satisfies both the sufficient conditions of Lemma~\ref{sup_and_inf}(1) and~(2): this strategy will be used to prove Theorem~\ref{nullmodification_thm} below. The requirement that the chain be countable can also be removed, by an argument involving the following real analysis lemma.

\begin{lemma} \label{countable_subset} 
Totally ordered sets have the following properties.
\begin{enumerate}
\item Any $\calT \subseteq \R$ contains a countable subset $\calC$ such that for $a \in \calT$ and $\epsilon > 0$, there exist $b, c \in \calC$ with $b \leq a \leq c$ and $c - b < \epsilon$. 
\item Consider a totally ordered set $\calT$ and a strictly increasing function $\mu : \calT \rightarrow \R$. Then $\calT$ contains a countable subset $\calC$ such that for $A \in \calT$ and $\epsilon > 0$, there exist $B, C \in \calC$ with $B \leq A \leq C$ and $\mu(C) - \mu(B) < \epsilon$.
\end{enumerate}
\end{lemma}

\begin{proof}
For 1, let $\calC_0$ be a countable, dense subset of $\calT$. (Such a subset exists since any subset of the reals is separable. Standard proofs of this invoke the axiom of countable choice.) Define 
\begin{align*}
\calC_1 &:= \{ c \in \calC : \exists \epsilon > 0 \mbox{ with } (c - \epsilon/2,c) \cap \calC_0 = \emptyset \}, \mbox{ and } \\ 
\calC_2 &:= \{ c \in \calC : \exists \epsilon > 0 \mbox{ with } (c, c + \epsilon/2) \cap \calC_0 = \emptyset \}.
\end{align*}
Then $\calC_1$ and $\calC_2$ are both countable, since there cannot be an uncountable number of pairwise disjoint intervals of non-zero width contained in $\R$ (each must contain a distinct rational). Thus $\calC := \calC_0 \cup \calC_1 \cup \calC_2$ is a countable set with the claimed property.

Claim~2 then follows by applying Claim~1 to $\mu(\calT)$, and noting $\mu$ is one-to-one and order preserving, as a consequence of being strictly increasing.
\qed \end{proof}

Lemma~\ref{nullmodification_oneside} can now be generalised as follows. The following theorem and its proof refer to equivalence classes in $\calP(\N)$: these are elements of the Boolean quotient $\calP(\N) / \calN$.

\begin{theorem} \label{nullmodification_thm}
Given a chain $\calT \subset \calF$, there exists a map $\phi : \calT_* \cap \calF \rightarrow \calF$ such that:
\begin{enumerate}
\item for all $A \in \calT_* \cap \calF$, $\phi(A) \triangle A \in \calN$, 
\item for all $A, B \in \calT_* \cap \calF$, 
\begin{align*}
[A] = [B] &\iff \phi(A) = \phi(B) \mbox{ and } \\
[A] < [B] &\iff \phi(A) \subset \phi(B),
\end{align*}
\item for all $\calS \subseteq \calT_* \cap \calF$, 
\begin{enumerate}
\item $\bigcup \phi(\calS) \in \calF$ with $\nu(\bigcup \phi(\calS)) = \sup \{ \nu(A) : A \in \phi(\calS) \}$, and
\item $\bigcap \phi(\calS) \in \calF$ with $\nu(\bigcap \phi(\calS)) = \inf \{ \nu(A) : A \in \phi(\calS) \}$,
\end{enumerate}
\item for all $A \in \phi(\calT_* \cap \calF)_*$ and $N \in \N$, $\nu_N(A) \leq \nu(A)$, 
\item if there exists a second map $\rho : \calT_* \cap \calF \rightarrow \calF$ such that $\rho(A) \triangle A \in \calN$ for all $A \in \calT_* \cap \calF$, then for all $\calS \subseteq \calT_* \cap \calF$, 
\begin{enumerate}
\item if $\bigcup \rho(\calS) \in \calF$ with $\nu(\bigcup \rho(\calS)) = \sup \{ \nu(A) : A \in \rho(\calS) \}$, then 
\[
\left( \bigcup \phi(\calS) \right) \triangle \left( \bigcup \rho(\calS) \right) \in \calN, \mbox{ and }
\]
\item if $\bigcap \rho(\calS) \in \calF$ with $\nu(\bigcap \rho(\calS)) = \inf \{ \nu(A) : A \in \rho(\calS) \}$, then 
\[
\left( \bigcap \phi(\calS) \right) \triangle \left( \bigcap \rho(\calS) \right) \in \calN.
\]
\end{enumerate}
\end{enumerate}
\end{theorem}

\begin{proof}
Without loss of generality, suppose $\calT = \calT_* \cap \calF$. (No generality is lost because for any chain $\calT \subset \calF$, the chain $\calS := \calT_* \cap \calF$ has the property $\calS = \calS_* \cap \calF$. Moreover, if the lemma holds for $\calS$, it holds for $\calT$.) 

Let $[\calC]$ be the countable subchain of $[\calT] := \{ [A] : A \in \calT \}$ obtained by applying Lemma~\ref{countable_subset}(2) to the totally ordered set $[\calT]$ and the strictly increasing function $\nu: [\calT] \rightarrow \R$ given by $\nu[A] := \nu(A)$. Then let $\calC \subseteq \calT$ be obtained by selecting exactly one element of $\calT$ from each of the equivalence classes in $[\calC]$. (This implicitly invokes the axiom of countable choice, in general.) 

Construct a map $\psi : \calC^c \rightarrow \calF$ as in the proof of Lemma~\ref{nullmodification_oneside}, and then a map $\psi^{\prime} : \psi(\calC^c)^c \rightarrow \calF$ also as in the proof of Lemma~\ref{nullmodification_oneside}. Define $\phi^{\prime}(A) := \psi^{\prime}(\psi(A^c)^c)$ for each $A \in \calC$. 

The map $\phi^{\prime}$ inherits Properties~2, 3a, and 4 of Lemma~\ref{nullmodification_oneside} from $\psi^{\prime}$ and $\psi$. Property~1 of that lemma must be weakened to $\phi^{\prime}(A) \triangle A$ for all $A \in \calC$, because $\phi^{\prime}$ effectively adds a null set to $A$ and then removes a null set. However, Property~3b of that lemma can be strengthened to the unconditional claim $\bigcap \phi^{\prime}(\calS) \in \calF$ with $\nu(\bigcap \phi^{\prime}(\calS)) = \inf \{ \nu(A) : A \in \phi^{\prime}(\calS) \}$ for all $\calS \subseteq \calC$. This follows because $\bigcup \psi(\calS^c) \in \calF$ with $\nu(\bigcup \psi(\calS^c)) = \sup \{ \nu(A): A \in \psi(\calS^c) \}$, by Lemma~\ref{nullmodification_oneside}(3a) as it applies to $\psi$. Hence the condition of Lemma~\ref{nullmodification_oneside}(3b), as it applies to $\psi^{\prime}$, is satisfied for any $\calS \subseteq \calC$, that is $\bigcap \psi(\calS^c)^c \in \calF$ with
\[
\nu(\bigcap \psi(\calS^c)^c) = \inf \{ \nu(A): A \in \psi(\calS^c)^c \}.
\]
Property~5 will be discussed later.

Define
\[
\phi(A) := \bigcup \{ \phi^{\prime}(B) : B \in \calC, [B] \leq [A] \}
\]
for each $A \in \calT$. Then $\phi$ is an extension of $\phi^{\prime}$, because for any $A, B \in \calC$ with $[B] \leq [A]$, one must have $B \subseteq A$, since $\calC$ contains at most one element from each equivalence class. Hence $\phi(A) = \phi^{\prime}(A)$.

Note that for any $A, B \in \calT$,
\begin{align*}
[A] = [B] &\implies \phi(A) = \phi(B) \mbox{ and } \\
[A] < [B] &\implies \phi(A) \subseteq \phi(B).
\end{align*}
This is immediate from the definition of $\phi$, and entails that $\phi$ is order preserving. 

To show~1, consider $A \in \calT$ and $\epsilon > 0$. There exist $B, C \in \calC$ with $[B] \leq [A] \leq [C]$ and $\nu(C) - \nu(B) < \epsilon / 2$. If $[A] = [B]$ then $A \triangle B \in \calN$ and $\phi(A) = \phi^{\prime}(B)$. Moreover, $\phi^{\prime}(B) \triangle B \in \calN$. Hence $\phi(A) \triangle A \in \calN$. Similarly, if $[A] = [C]$ then $\phi(A) \triangle A \in \calN$. Suppose $[B] < [A] < [C]$, implying $B \subset A \subset C$. Then $\phi^{\prime}(B) \subseteq \phi(A) \subseteq \phi^{\prime}(C)$, since $\phi$ is order preserving, and $\nu(\phi^{\prime}(C)) - \nu(\phi^{\prime}(B)) < \epsilon / 2$, using Lemma~\ref{calFproperties}(6). It follows that
\[
\phi(A) \triangle A \subseteq (\phi^{\prime}(B) \triangle B) \cup (\phi^{\prime}(C) \setminus \phi^{\prime}(B)) \cup (C \setminus B),
\]
and
\[
\nu^+(\phi(A) \triangle A) \leq \nu(\phi(B) \triangle B) + \nu(\phi^{\prime}(C)) - \nu(\phi^{\prime}(B)) + \nu(C) - \nu(B) < \epsilon.
\]
Letting $\epsilon \rightarrow 0$ gives $\phi(A) \triangle A \in \calN$. This also implies $\phi(\calT) \subset \calF$, using Lemma~\ref{calFproperties}(6).

For 2, note
\[
\phi(A) = \phi(B) \implies [\phi(A)] = [\phi(B)] \implies [A] = [B],
\]
since $A \in [\phi(A)]$ and $B \in [\phi(B)]$. This in turn gives 
\[
[A] < [B] \implies \phi(A) \neq \phi(B) \implies \phi(A) \subset \phi(B)
\]
since it is already established that $[A] < [B] \implies \phi(A) \subseteq \phi(B)$. Moreover,
\begin{eqnarray*}
\phi(A) \subset \phi(B) &\implies& \phi(A) \neq \phi(B) \mbox{ and } \phi(B) \not\subset \phi(A) \\
&\implies& [A] \neq [B] \mbox{ and } [B] \not< [A] \\
&\implies& [A] < [B].
\end{eqnarray*}

For 3a, consider $\calS \subseteq \calT$. If $\bigcup \phi(\calS) = \phi(A)$ for any $A \in \calS$, then  the result holds trivially, so assume $\bigcup \phi(\calS) \supset \phi(A)$ for all $A \in \calS$. But then for any $A \in \calS$, there is $B \in \calS$ with $\phi(A) \subset \phi(B) \subset \bigcup \phi(\calS)$ and $C \in \calC$ with $\phi(B) \subseteq \phi^{\prime}(C) \subset \bigcup \phi(\calS)$. Define $\calB := \{ C \in \calC : \phi^{\prime}(C) \subset \bigcup \phi(\calS) \}$, then $\bigcup \phi(\calS) = \bigcup \phi^{\prime}(\calB) \in \calF$ and 
\begin{eqnarray*}
\nu \bigg( \bigcup \phi(\calS) \bigg) &=& \nu \bigg( \bigcup \phi^{\prime}(\calB) \bigg) \\
&=& \sup \{ \nu(\phi^{\prime}(C)) : C \in \calB \} \\
&=& \sup \{ \nu(\phi(A)) : A \in \calS \}.
\end{eqnarray*}
Property~3b follows by a similar argument.

For 4, consider $A \in \phi(\calT)_*$, $N \in \N$ and $\epsilon > 0$. Consider the case $A := \bigcup \phi(\calS)$ for some $\calS \subseteq \calT$. If $\nu_N(A) = 0$ then trivially $\nu_N(A) \leq \nu(A)$, so assume $A$ contains some element of $\{ 1, \ldots, N \}$. Then there is $B \in \calS$ such that $\nu_N(\phi(B)) = \nu_N(A)$, since each element of $A \cap \{ 1, \ldots, N \}$ must be contained in some member of $\phi(\calS)$. Moreover, there exists $C \in \calC$ such that $[B] \leq [C]$ and $\nu(C) - \nu(B) \leq \epsilon$. Putting this all together,
\[
\nu_N(A) = \nu_N(\phi(B)) \leq \nu_N(\phi(C)) \leq \nu(\phi(C)) \leq \nu(\phi(B)) + \epsilon \leq \nu(A) + \epsilon.
\]
Letting $\epsilon \rightarrow 0$ gives $\nu_N(A) \leq \nu(A)$.

On the other hand, if $A := \bigcap \phi(\calS)$ for some $\calS \subseteq \calT$, then there is $B \in \calS$ with $\nu(\phi(B)) - \nu(A) < \epsilon$, by Property 3b. Hence 
\begin{eqnarray*}
\nu_N(A) - \nu(A) &=& \big( \nu_N(A) - \nu_N(\phi(B)) \big) + \big( \nu_N(\phi(B)) - \nu(\phi(B)) \big) \\
&& + \big( \nu(\phi(B)) - \nu(A) \big) \\
&\leq& 0 + 0 + \epsilon, 
\end{eqnarray*}
where the first summand is less than or equal 0 because $A \subseteq \phi(B)$, and the second summand is less than or equal 0 by the first case. Letting $\epsilon \rightarrow 0$ gives $\nu_N(A) \leq \nu^+(A)$.

Property~5a follows by essentially the same argument as Theorem~\ref{nullmodification_oneside}(5), merely replacing $\psi$ with $\phi$. Property~5b also follows by a similar argument, but using Property~3b instead of 3a.
\qed \end{proof}


\section{The space $[\calF]$}
\label{quotient_space}

This section uncovers a useful property of the set $\calF$, specifically that the quotient map $\xi: \calP(\N) \rightarrow \calP(\N) / \calN$ maps $\calF$ to a monotone class of $\calP(\N) / \calN$. 

Recall the notation
\[
[\calF] := \{ [A]  \in \calP(\N) / \calN : A \in \calF \}.
\]
The set function $\nu: \calF \rightarrow [0,1]$ induces a corresponding function on $[\calF]$ defined by
\[
\nu[A] := \nu(A),
\]
This function is well defined, since if $[A]  = [B]$ and $A \in \calF$, then $A \triangle B \in \calN$ and $B \in \calF$ with $\nu(A) = \nu(B)$ by Properties~5 and~6 of Lemma~\ref{calFproperties}. Similarly the set functions $\nu^+: \calP(\N) \rightarrow [0,1]$ and $\nu^- : \calP(\N) \rightarrow [0,1]$ induce functions on $\calP(\N) / \calN$ defined by
\[
\nu^+[A] := \nu^+(A) \mbox{ and } \nu^-[A] := \nu^-(A).
\]
These set functions are well defined as a consequence of the following lemma.

\begin{lemma} \label{dance}
If $A \triangle B \in \calN$ then $\nu^+(A) = \nu^+(B)$ and $\nu^-(A) = \nu^-(B)$.
\end{lemma}
\begin{proof}
For any $N \in \N$,
\[
\nu_N(A \cup B) = \nu_N(A) + \nu_N(B \setminus A) = \nu_N(B) + \nu_N(A \setminus B).
\]
Consequently,
\[
| \nu_N(A) - \nu_N(B) | = | \nu_N(A \setminus B) - \nu_N(B \setminus A) | \leq \nu_N(A \setminus B) + \nu_N(A \setminus B) = \nu_N(A \triangle B).
\]
But then
\[
\nu_N(A) \leq \nu_N(B) + | \nu_N(A) - \nu_N(B) | \leq \nu_N(B) + \nu_N(A \triangle B).
\]
Taking the $\limsup$ as $N \rightarrow \infty$ gives
\[
\nu^+(A) \leq \nu^+(B) + \nu(A \triangle B) = \nu^+(B).
\]
Similarly $\nu^+(B) \leq \nu^+(A)$ and thus $\nu^+(A) = \nu^+(B)$. To show the corresponding result for $\nu^-$, note that $A \triangle B \in \calN$ implies $A^c \triangle B^c \in \calN$, and so
\[
\nu^-(A) = 1 - \nu^+(A^c) = 1 - \nu^+(B^c) = \nu^-(B).
\]
\qed \end{proof}
Note that $\nu^+[A] = \nu^-[A]$ if and only if $[A]  \in [\calF]$ and $\nu[A] = \nu^+[A]$, a property inherited from the corresponding set functions on $\calP(\N)$. 

As Lemma \ref{dance} and the preceding discussion suggest, many of the properties of $\calF$ carry over naturally to $[\calF]$, in some cases with simplified or stronger statements. These are summarized in the lemma below, the proof of which is straightforward and omitted. Note in particular that the final claim of Statement~5 below is stronger than the corresponding claim in Lemma~\ref{calFproperties}, because $\nu[A] = 0 \iff [A] = [\emptyset]$ for $A \in \calF$.

Property~7 below refers to disjoint elements in $\calP(\N) / \calN$. This is conventional terminology for elements of an abstract Boolean algebra that have a meet of 0, but nevertheless the following definition may clarify the meaning of `disjoint' in the present context.

\begin{definition}
Equivalence classes $[A], [B] \in \calP(\N) / \calN$ are said to be {\em disjoint} if $[A] \wedge [B] = [\emptyset]$, or equivalently $A \cap B \in \calN$.
\end{definition}
Note the two definitions are equivalent because $[A] \wedge [B] = [A \cap B]$, and the latter is equal to $[\emptyset]$ if and only if $A \cap B \in \calN$.

\begin{lemma} \label{quotient_properties}
The collection $[\calF]$ and functions $\nu^+$, $\nu^-$ and $\nu$ have the following properties.
\begin{enumerate}
\item $[\emptyset], [\N] \in [\calF]$, with $\nu[\emptyset] = 0$ and $\nu[\N ] = 1$.
\item For all $[A] \in [\calF]$, $[A]^c \in [\calF]$ with $\nu([A]^c) = 1 - \nu[A]$.
\item For all $[A], [B] \in [\calF]$, $[A] \vee [B] \in [\calF]$ if and only if $[A] \wedge [B] \in [\calF]$, and if either is true then $\nu([A] \vee [B]) = \nu[A] + \nu[B] - \nu([A] \wedge [B])$.
\item If $[A], [B] \in \calP(\N) / \calN$, then $\nu^+([A] \vee [B]) \leq \nu^+[A] + \nu^+[B]$. If $[A]$, $[B]$, and $[A] \vee [B]$ are all in $\calF$, then $\nu([A] \vee [B]) \leq \nu[A] + \nu[B]$.
\item For $[A], [B] \in \calP(\N) / \calN$ such that $[A] \leq [B]$, 
\begin{enumerate}
\item $\nu^+[A] \leq \nu^+[B]$,  
\item $\nu^+([B] - [A]) \geq \nu^+[B] - \nu^+[A]$, 
\item $\nu^-[A] \leq \nu^-[B]$, and 
\item $\nu^-([B] - [A]) \leq \nu^-[B] - \nu^-[A]$. 
\end{enumerate}
If in addition $[A], [B] \in [\calF]$, then 
\begin{enumerate}
\item $\nu[A] \leq \nu[B]$, 
\item $[B] - [A] \in [\calF]$,
\item $\nu([B] - [A]) = \nu[B] - \nu[A]$, and
\item $\nu[A] = \nu[B] \iff [A] = [B]$.
\end{enumerate}
\item For $[A] \in \calP(\N) / \calN$, $\nu^+[A] = 0 \iff [A] = [\emptyset]$. 
\item For pairwise disjoint elements $[A_1], \ldots, [A_K] \in [\calF]$, $\vee_{k=1}^K [A_k] \in [\calF]$ with 
\[
\nu(\vee_{k=1}^K [A_k]) = \sum_{k=1}^K \nu[A_k].
\] 
\end{enumerate}
\end{lemma}

The properties related to countable additivity - Properties 8 to 10 of Lemma \ref{calFproperties} - are not carried over to this setting because Theorem~\ref{calF_quotient_complete} below establishes something stronger: that $\nu$ is countably additive on pairwise disjoint sequences in $[\calF]$.

\begin{theorem} \label{calF_quotient_complete}
For pairwise disjoint equivalence classes $\{ [ A_k ]  \}_{k=1}^{\infty} \subset [\calF]$, the representative sets $\{ A_k \}_{k=1}^{\infty} \subset \calF$ can be chosen so that
\begin{enumerate}
\item $\{ A_k \}_{k=1}^{\infty}$ are pairwise disjoint, and
\item $A := \cup_{k=1}^{\infty} A_k \in \calF$ with $\nu(A) = \sum_{k=1}^{\infty} \nu(A_k)$.
\end{enumerate}
Moreover, $[ A ]$ is the least upper bound of $\{ [A_k] \}_{k=1}^{\infty}$ in $\calP(\N) / \calN$. 
\end{theorem}

\begin{proof}
Consider pairwise disjoint $\{ [ A_k ]  \}_{k=1}^{\infty} \subset [\calF]$. Then $A_j \cap A_k \in \calN$ for $j, k \in \N$ with $j \neq k$, hence $A_k \cap (\cup_{j=1}^{k-1} A_j) \in \calN$ and $[ A_k ]  = [ A_k \setminus (\cup_{j=1}^{k-1} A_j) ]$. Thus one may instead choose $\{ A_k \setminus (\cup_{j=1}^{k-1} A_j) \}_{k=1}^{\infty}$ to be the representative sets from each equivalence class to ensure Claim~1 holds. One may then use Corollary~\ref{nullmodification_sequence} to replace each representative set with a subset in the same equivalence class to ensure Claim~2 holds. 

To see that $[ A ]$ is the least upper bound  of $\{ [A_k]  \}_{k=1}^{\infty}$ in $\calP(\N) / \calN$, first note $[ A ]$ is an upper bound for this set because $A_k \subseteq A$ for all $k$, and the quotient map respects the partial order. Suppose there is some other $B \in \calP(\N)$ such that $[ A_k ]  \leq [ B ]$ for all $k$. Define $A_k^{\prime} := B \cap A_k$ for each $k$, and $A^{\prime} := \cup_{k=1}^{\infty} A_k^{\prime}$. Then $[ A^{\prime} ]  = [ B \cap A ]  = [ B ]  \wedge [ A ]$, giving $[ A^{\prime} ] \leq [B]$. Moreover, since both $[A]$ and $[B]$ are upper bounds for $[A_k]$, so is $[ A^{\prime} ]$, which implies $A_k \setminus A^{\prime} \in \calN$ for all $k \in \N$. Now, for any $K \in \N$,
\begin{eqnarray*}
\nu^+(A \setminus A^{\prime}) &\leq& \nu^+(A \setminus \cup_{k=1}^K A_k) + \nu^+( \cup_{k=1}^K A_k \setminus A^{\prime}) \\
&=& \nu(\cup_{k=K+1}^{\infty} A_k) \\
&=& \lim_{K \rightarrow \infty} \sum_{k = K+1}^{\infty} \nu(A_k).
\end{eqnarray*}
Letting $K \rightarrow \infty$ gives $\nu^+(A \setminus A^{\prime}) = 0$ and hence $[A] \leq [A^{\prime}] \leq [B]$.
\qed \end{proof}

This theorem has the following useful corollary.

\begin{corollary} \label{countableadditivity}
$[\calF]$ is a monotone class in $\calP(\N) / \calN$, on which $\nu$ is countably additive.
\end{corollary}

\begin{proof}
If $B_1 \leq B_2 \leq \ldots$ is a non-decreasing sequence in $[\calF]$, Theorem~\ref{calF_quotient_complete} implies the sequence has a supremum $B \in [\calF]$. Similarly, if $C_1 \geq C_2 \geq \ldots$ is a non-increasing sequence in $[\calF]$, then $C_1^{\prime} \leq C_2^{\prime} \leq \ldots$ is a non-decreasing sequence in $[\calF]$ with supremum $C^{\prime} \in [\calF]$, and then $C$ is the infimum for $\{ C_k \}_{k=1}^{\infty}$. Thus $[\calF]$ is a monotone class. Countable additivity of $\nu$ on $[\calF]$ follows from Theorem~\ref{calF_quotient_complete}(2).
\qed \end{proof}

As a consequence of this corollary and the monotone class theorem for Boolean algebras, $[\calF]$ contains the countably complete Boolean algebras generated by any Boolean algebra contained in $[\calF]$. 

\bibliography{mybibfile}

\section*{Acknowledgements} 
The authors are grateful to the Australian Research Council Centre of Excellence for Mathematical and Statistical Frontiers (CE140100049) for their support.

\end{document}